\documentclass{article}
\usepackage[left=25 mm, bottom=25 mm, right=25 mm, top=25 mm]{geometry}
\usepackage{amsmath, amssymb, amsthm, csquotes, hyperref, multirow, float, tikz, subfig, blkarray, enumitem}
\usetikzlibrary{arrows, decorations.markings}
\newtheorem{theorem}{Theorem}[section]
\newtheorem{lemma}[theorem]{Lemma}
\newtheorem{conj}[theorem]{Conjecture}
\newtheorem{remark}[theorem]{Remark}
\newtheorem{prop}[theorem]{Proposition}

\usepackage{titlesec}
\titleformat*{\section}{\small\bfseries}
\titleformat*{\subsection}{\small\bfseries}

\newcommand{\rk}{\operatorname{rank}}
\newcommand{\U}{\operatorname{\bf u}}
\newcommand{\V}{\operatorname{\bf v}}
\newcommand{\W}{\operatorname{\bf w}}
\newcommand{\Z}{\operatorname{\bf z}}
\newcommand{\X}{\operatorname{\bf x}}
\newcommand{\Y}{\operatorname{\bf y}}
\newcommand{\0}{\operatorname{\bf 0}}

\hypersetup{colorlinks=true, linkcolor=blue, citecolor=blue}
\usepackage{authblk}

\title{Sign patterns that allow algebraic positivity}
\author{Sunil Das\footnote{ Email address: {\tt sunil.das@isid.ac.in}}}
\affil{\small Indian Statistical Institute, New Delhi-110016, India}

\begin{document}
\maketitle
\noindent\hrulefill\vspace*{-3 mm}
\begin{abstract}
\noindent A real square matrix is algebraically positive if there exists a real polynomial $f$ such that $f(A)$ is a positive matrix. In this paper, we give a sufficient condition for a sign pattern matrix to allow algebraic positivity, and give some methods to construct higher-order algebraically positive matrices from some lower-order algebraically positive matrices. We also propose two conjectures related to the problem of allowing algebraic positivity.
\end{abstract}\vspace*{-4 mm}
\noindent\hrulefill
\vspace*{2 mm}


\section{Introduction}
A real matrix is said to be positive if its every entry is positive. Kirkland, Qiao, and Zhan introduced the concept of algebraically positive matrices in 2016. A real square matrix $A$ is algebraically positive if there exists a real polynomial $f$ such that $f(A)$ is a positive matrix. They gave the following characterization of algebraically positive matrices.

\begin{theorem}[\cite{KQZ16}]\label{thm1.1}
A real square matrix is algebraically positive if and only if it has a simple real eigenvalue and corresponding left and right positive eigenvectors.
\end{theorem}

A sign pattern matrix is a matrix whose every entry is a symbol from the set $\{+,-,0\}$. The qualitative class of a matrix $A$, denoted by $Q(A)$, is defined by the set of all real matrices obtained from $A$ by replacing \enquote*{$+$} by some positive number, \enquote*{$-$} by some negative number and \enquote*{0} by the zero number. A sign pattern matrix $A$ allows a property $P$ if at least one matrix in $Q(A)$ has the property $P$ and requires property $P$ if all matrices in $Q(A)$ have the property $P$.

Kirkland, Qiao, and Zhan \cite{KQZ16} characterized symmetric tridiagonal sign pattern matrices that allow algebraic positivity. Later, Das and Bandopadhyay \cite{DB19} generalized this result to find all tree sign pattern matrices allowing algebraic positivity, and Abagat and Pelejo \cite{AP19} described all 3-by-3 sign pattern matrices that allow algebraic positivity. In this paper, we present a class of sign pattern matrices that allow algebraic positivity and give some methods to construct higher-order algebraically positive matrices from some lower-order algebraically positive matrices.

Let $D$ be a directed graph with the vertex set $V(D)=\{v_1,v_2,\ldots,v_r\}$. We denote the arc from vertex $v_i$ to vertex $v_j$ by $(\overrightarrow{v_i,v_j})$. If $V\subseteq V(D)$, then the subgraph of $D$ induced by $V$ is the directed graph obtained from $D$ by removing all vertices in $V(D)\setminus V$ and the arcs incident to them. A directed cycle of length $m$ in $D$ is a sequence $(\overrightarrow{v_{i_1},v_{i_2}}), (\overrightarrow{v_{i_2},v_{i_3}}), \ldots , (\overrightarrow{v_{i_{m-1}},v_{i_m}}), (\overrightarrow{v_{i_m},v_{i_1}})$ of arcs such that all the vertices $v_{i_1},v_{i_2},\ldots,v_{i_m}$ are distinct. A sequence $(\overrightarrow{v_{i_1},v_{i_2}}), (\overrightarrow{v_{i_2},v_{i_3}}), \ldots , (\overrightarrow{v_{i_{m-1}},v_{i_m}})$ of arcs, where all the vertices $v_{i_1},v_{i_2},\ldots,v_{i_m}$ are distinct, is said to be a directed path from $v_{i_1}$ to $v_{i_m}$. The directed graph $D$ is said to be strongly connected, if for any two distinct $i,j$, there is a directed path from $v_i$ to $v_j$. Moreover, $D$ is said to be minimally strongly connected, if it does not remain strongly connected after the deletion of an arc. 

Let us recall from \cite{HL14} that a sign pattern matrix $X$ is said to be a subpattern of $A$ if $X$ is obtained from $A$ by replacing some (possibly none) of its nonzero entries with 0. In this case, $A$ is said to be a super-pattern of $X$. The digraph of a sign pattern matrix $A$ of order $n$, denoted by $D(A)$, is defined by a directed graph with vertex set $\{1,2,\ldots,n\}$, where $(\overrightarrow{i,j})$ is an arc of $D(A)$ if and only if $a_{ij}\neq0$.

In Section \ref{sec2}, we give some methods to construct higher-order algebraically positive matrices from some lower-order algebraically positive matrices, and in Section \ref{sec3}, we give a sufficient condition for a sign pattern matrix to allow algebraic positivity. We follow a constructive approach to prove it. In Section \ref{sec4}, we give an example to illustrate the methodology of that approach. Finally, we conclude the paper with two conjectures related to the problem of allowing algebraic positivity.

\section{Constructions of algebraically positive matrices}\label{sec2}
Throughout this paper, ${\0}$ denotes a zero column vector, and $I$ denotes an identity matrix, whose orders will be clear from the context. Moreover, $\mathbb{C}$ denotes the set of all complex numbers.

The following lemma from \cite[p. 80]{HJ85} is useful to check whether an eigenvalue of a matrix is simple or not.

\begin{lemma}[\cite{HJ85}]\label{lem2.1}
Let $A$ be a square matrix and $\X,\Y\in\mathbb{C}^n\setminus\{\bf 0\}$ such that $A\X=\lambda \X$ and $\Y^*A=\lambda\Y^*$ for some $\lambda\in\mathbb{C}$. If $\lambda$ has geometric multiplicity $1$, then $\lambda$ is simple if and only if $\Y^*\X\neq0$.
\end{lemma}

Now we give some methods to construct higher-order algebraically positive matrices from some lower-order algebraically positive matrices using Theorem \ref{thm1.1} and Lemma \ref{lem2.1}.

\begin{theorem}\label{thm2.2}
Let $A$ be a real square matrix of order $n$ such that $a_{1j}<0$ for some $j\in\{1,2,\ldots,n\}$. Suppose that $\lambda>0$ is a simple eigenvalue of $A$, and $A\U=\lambda\U$ and $\V^TA=\lambda\V^T$ for some $\U,\V>\0$. Then $\lambda$ is a simple eigenvalue of $B$, where $B$ is of order $k+n$, given by
$$b_{pq}=\begin{cases}
\lambda-a_{1j},&\mbox{if }p=q\in\{1,2,\ldots,k\};\\
a_{1j},&\mbox{if }(p,q)\in\{(k,k+j),(k+1,1)\}\cup\{(1,2),(2,3),\ldots,(k-1,k)\};\\
a_{p-k,q-k},&\mbox{if }p,q\in\{k+1,k+2,\ldots,k+n\}\mbox{ such that }(p,q)\neq(k+1,k+j);\\
0,&\mbox{otherwise}.
\end{cases}$$
Moreover, a right eigenvector $\W$ and a left eigenvector $\Z$ of $B$ corresponding to $\lambda$ are given by:
$$w_i=\begin{cases}
u_j,&\mbox{if }1\leq i\leq k;\\
u_{i-k},&\mbox{if }k+1\leq i\leq k+n.
\end{cases} \qquad \qquad
z_i=\begin{cases}
v_1,&\mbox{if }1\leq i\leq k;\\
v_{i-k},&\mbox{if }k+1\leq i\leq k+n.
\end{cases}$$
Consequently, $B$ is algebraically positive.
\end{theorem}

\begin{proof}
Introducing the row and the column indices, we can write $B$ as
$$\begin{blockarray}{ccccccccccccc}
\begin{block}{c[ccccc|ccccccc]}
_1&\lambda -a_{1j}&a_{1j}&0&\cdots&0& 0&\cdots&\cdots&\cdots&\cdots&\cdots&0 \\
\vdots&0&\ddots&\ddots&\ddots&\vdots&\vdots&&&&&&\vdots \\
\vdots&\vdots&\ddots&\ddots&\ddots&0&\vdots&&&&&&\vdots \\
\vdots&\vdots&&\ddots&\ddots&a_{1j}& 0&\cdots&\cdots&\cdots&\cdots&\cdots&0 \\
_k&0&\cdots&\cdots&0&\lambda -a_{1j}&0&\cdots&0&a_{1j}&0&\cdots&0 \\
\cline{2-13}
_{k+1}& a_{1j}&0&\cdots&\cdots&0& a_{11}&\cdots&a_{1,j-1}&0&a_{1,j+1}&\cdots&a_{1n} \\
_{k+2}& 0&\cdots&\cdots&\cdots&0& a_{21}&\cdots&\cdots&a_{2j}&\cdots&\cdots&a_{2n} \\
\vdots& \vdots&&&&\vdots& \vdots&&&&&&\vdots \\
_{k+n}& 0&\cdots&\cdots&\cdots&0& a_{n1}&\cdots&\cdots&a_{nj}&\cdots&\cdots&a_{nn} \\
\end{block}\\[-2 mm]
&_1&\ldots&\ldots&\ldots&_k&_{k+1}&\ldots&\ldots&_{k+j}&\ldots&\ldots&_{k+n} \\
\end{blockarray}.$$

Clearly, $B\W=\lambda\W$ and $\Z^TB=\lambda\Z^T$ and thus $\lambda$ is an eigenvalue of $B$. Further, $\W$ and $\Z$ are positive vectors. By elementary column operations, we have the following:
$$B-\lambda I\xrightarrow{C_{k+j}\leftarrow C_{k+j}+C_1+C_2+\cdots+C_k}
\left[\begin{array}{ccccc|c}
-a_{1j}&a_{1j}&0&\cdots&0&\0^T \\
0&\ddots&\ddots&\ddots&\vdots&\vdots \\
\vdots&\ddots&\ddots&\ddots&0&\vdots \\
\vdots&&\ddots&\ddots&a_{1j}&\vdots \\
0&\cdots&\cdots&0&-a_{1j}&\0^T \\
\hline
a_{1j}&0&\cdots&\cdots&0&\multirow{2}{*}{$A-\lambda I$} \\
\0&\cdots&\cdots&\cdots&\0&
\end{array}\right]$$
Since $\lambda$ is a simple eigenvalue of $A$, $\rk(A-\lambda I)=n-1$ and thus $\rk(B-\lambda I)=n+k-1$. Therefore by Lemma \ref{lem2.1}, $\lambda$ is a simple eigenvalue of $B$, and hence by Theorem \ref{thm1.1}, $B$ is algebraically positive.
\end{proof}

Applying similar arguments as in the proof of Theorem \ref{thm2.2}, we can prove the following theorems of this section.

\begin{theorem}\label{thm2.3}
Let $A$ be a real square matrix of order $n$ such that $a_{1j}>0$ for some $j\in\{1,2,\ldots,n\}$. Suppose that $\lambda>0$ is a simple eigenvalue of $A$, and $A\U=\lambda\U$ and $\V^TA=\lambda\V^T$ for some $\U,\V>\0$. Then $\lambda$ is a simple eigenvalue of $B$, where $B$ is of order $k+n$, given by
$$b_{pq}=\begin{cases}
2\lambda,&\mbox{if }p=q\in\{1,2,\ldots,k\};\\
-\lambda,&\mbox{if }(p,q)\in\{(k,k+j)\}\cup\{(1,2),(2,3),\ldots,(k-1,k)\};\\
-a_{1j},&\mbox{if }(p,q)=(k+1,1);\\
a_{p-k,q-k},&\mbox{if }p,q\in\{k+1,k+2,\ldots,k+n\}\mbox{ such that }(p,q)\neq(k+1,k+j);\\
2a_{1j},&\mbox{if }(p,q)=(k+1,k+j);\\
0,&\mbox{otherwise}.
\end{cases}$$ 
Moreover, a right eigenvector $\W$ and a left eigenvector $\Z$ of $B$ corresponding to $\lambda$ are given by:
$$w_i=\begin{cases}
u_j,&\mbox{if }1\leq i\leq k;\\
u_{i-k},&\mbox{if }k+1\leq i\leq k+n.
\end{cases} \qquad \qquad
z_i=\begin{cases}
\frac{a_{1j}v_1}{\lambda},&\mbox{if }1\leq i\leq k;\\
v_{i-k},&\mbox{if }k+1\leq i\leq k+n.
\end{cases}$$
Consequently, $B$ is algebraically positive.
\end{theorem}

\begin{remark}
Introducing the row and the column indices, we can write $B$ in Theorem \ref{thm2.3} as
$$\begin{blockarray}{ccccccccccccc}
\begin{block}{c[ccccc|ccccccc]}
_1&2\lambda&-\lambda&0&\cdots&0& 0&\cdots&\cdots&\cdots&\cdots&\cdots&0 \\
\vdots&0&\ddots&\ddots&\ddots&\vdots&\vdots&&&&&&\vdots \\
\vdots&\vdots&\ddots&\ddots&\ddots&0&\vdots&&&&&&\vdots \\
\vdots&\vdots&&\ddots&\ddots&-\lambda& 0&\cdots&\cdots&\cdots&\cdots&\cdots&0 \\
_k&0&\cdots&\cdots&0&2\lambda&0&\cdots&0&-\lambda&0&\cdots&0 \\
\cline{2-13}
_{k+1}& -a_{1j}&0&\cdots&\cdots&0& a_{11}&\cdots&a_{1,j-1}&2a_{1j}&a_{1,j+1}&\cdots&a_{1n} \\
_{k+2}& 0&\cdots&\cdots&\cdots&0& a_{21}&\cdots&\cdots&a_{2j}&\cdots&\cdots&a_{2n} \\
\vdots& \vdots&&&&\vdots& \vdots&&&&&&\vdots \\
_{k+n}& 0&\cdots&\cdots&\cdots&0& a_{n1}&\cdots&\cdots&a_{nj}&\cdots&\cdots&a_{nn} \\
\end{block}\\[-2 mm]
&_1&\ldots&\ldots&\ldots&_k&_{k+1}&\ldots&\ldots&_{k+j}&\cdots&\cdots&_{k+n} \\
\end{blockarray}.$$
\end{remark}

\begin{theorem}\label{thm2.4}
Let $0<\epsilon<1$, and let $A$ be a real square matrix of order $n$ such that $a_{11}>0$. If $\lambda>0$ is a simple eigenvalue of $A$ such that $A\U=\lambda\U$ and $\V^TA=\lambda\V^T$ for some $\U,\V>\0$, then the following statements are true.
\begin{enumerate}
\item For every $j\in\{1,2,\ldots,k\}$, $\lambda$ is a simple eigenvalue of $B$, where $B$ is of order $k+n-1$, given by
$$b_{pq}=\begin{cases}
\lambda,&\mbox{if }(p,q)\in\{(1,2),(2,3),\ldots,(k-1,k)\};\\
a_{11},&\mbox{if }(p,q)=(k,1);\\
a_{p-k+1,1},&\mbox{if }p\in\{k+1,k+2,\ldots,k+n-1\}\mbox{ and }q=j;\\
a_{p-k+1,q-k+1},&\mbox{if }p\in\{k,k+1,\ldots,k+n-1\}\mbox{ and }q\in\{k+1,k+2,\ldots,k+n-1\};\\
0,&\mbox{otherwise}.
\end{cases}$$  
Moreover, a right eigenvector $\W$ and a left eigenvector $\Z$ of $B$ corresponding to $\lambda$ are given by:
$$w_i=\begin{cases}
u_1,&\mbox{if }1\leq i\leq k-1;\\
u_{i-k+1},&\mbox{if }k\leq i\leq k+n-1.
\end{cases} \qquad \qquad
z_i=\begin{cases}
\frac{a_{11}v_1}{\lambda},&\mbox{if }1\leq i\leq j-1;\\
v_1,&\mbox{if }j\leq i\leq k-1;\\
v_{i-k+1},&\mbox{if }k\leq i\leq k+n-1.
\end{cases}$$
Consequently, $B$ is algebraically positive. 
\item For every $j\in\{1,2,\ldots,k-1\}$ and $s\in\{1,2,\ldots,k\}\setminus\{j+1\}$, $\lambda$ is a simple eigenvalue of $B$, where $B$ is of order $k+n-1$, given by
$$b_{pq}=\begin{cases}
\epsilon\lambda,&\mbox{if }(p,q)=(j,j+1);\\
\lambda,&\mbox{if }(p,q)\in\{(1,2),(2,3),\ldots,(k-1,k)\}\setminus\{(j,j+1)\};\\
a_{11},&\mbox{if }(p,q)=(k,1);\\
(1-\epsilon)\lambda,&\mbox{if }(p,q)=(j,s);\\
a_{p-k+1,1},&\mbox{if }p\in\{k+1,k+2,\ldots,k+n-1\}\mbox{ and }q=s;\\
a_{p-k+1,q-k+1},&\mbox{if }p\in\{k,k+1,\ldots,k+n-1\}\mbox{ and }q\in\{k+1,k+2,\ldots,k+n-1\};\\
0,&\mbox{otherwise}.
\end{cases}$$ 
If $s\leq j$, then a right eigenvector $\W$ and a left eigenvector $\Z$ of $B$ corresponding to $\lambda$ are given by:
$$w_i=\begin{cases}
u_1,&\mbox{if }1\leq i\leq k-1;\\
u_{i-k+1},&\mbox{if }k\leq i\leq k+n-1.
\end{cases} \qquad \qquad
z_i=\begin{cases}
\frac{a_{11}v_1}{\lambda},&\mbox{if }1\leq i\leq s-1;\\
\frac{v_1}{\epsilon},&\mbox{if }s\leq i\leq j;\\
v_1,&\mbox{if }j+1\leq i\leq k-1;\\
v_{i-k+1},&\mbox{if }k\leq i\leq k+n-1.
\end{cases}$$   
Consequently, $B$ is algebraically positive. Further, for some suitable $\epsilon$ satisfying $0<\epsilon<1$,
$$b_{js}z_j+\sum\limits_{i=1}^{n-1}b_{k+i,s}z_{k+i}=(1-\epsilon)\frac{\lambda v_1}{\epsilon}+\sum\limits_{i=1}^{n-1}a_{i+1,1}v_{i+1}=(1-\epsilon)\frac{\lambda v_1}{\epsilon}+(\lambda-a_{11})v_1=\left(\frac{\lambda}{\epsilon}-a_{11}\right) v_1>0.$$
If $s\geq j+2$, then a right eigenvector $\W$ and a left eigenvector $\Z$ of $B$ corresponding to $\lambda$ are given by:
$$w_i=\begin{cases}
u_1,&\mbox{if }1\leq i\leq k-1;\\
u_{i-k+1},&\mbox{if }k\leq i\leq k+n-1.
\end{cases} \qquad \qquad
z_i=\begin{cases}
\frac{a_{11}v_1}{\lambda},&\mbox{if }1\leq i\leq j;\\
\frac{\epsilon a_{11}v_1}{\lambda},&\mbox{if }j+1\leq i\leq s-1;\\
v_1,&\mbox{if }s\leq i\leq k-1;\\
v_{i-k+1},&\mbox{if }k\leq i\leq k+n-1.
\end{cases}$$
Consequently, $B$ is algebraically positive. Further, for some suitable $\epsilon$ satisfying $0<\epsilon<1$,
$$b_{js}z_j+\sum\limits_{i=1}^{n-1}b_{k+i,s}z_{k+i}=(1-\epsilon)a_{11} v_1+\sum\limits_{i=1}^{n-1}a_{i+1,j}v_{i+1}=(1-\epsilon)a_{11} v_1+(\lambda-a_{11})v_1=(\lambda-\epsilon a_{11})v_1>0.$$
\item For every $s\in\{2,3,\ldots,k\}$, $\lambda$ is a simple eigenvalue of $B$, where $B$ is of order $k+n-1$, given by
$$b_{pq}=\begin{cases}
\lambda,&\mbox{if }(p,q)\in\{(1,2),(2,3),\ldots,(k-1,k)\};\\
\epsilon a_{11},&\mbox{if }(p,q)=(k,1);\\
(1-\epsilon)a_{11},&\mbox{if }(p,q)=(k,s);\\
a_{p-k+1,1},&\mbox{if }p\in\{k+1,k+2,\ldots,k+n-1\}\mbox{ and }q=s;\\
a_{p-k+1,q-k+1},&\mbox{if }p\in\{k,k+1,\ldots,k+n-1\}\mbox{ and }q\in\{k+1,k+2,\ldots,k+n-1\};\\
0,&\mbox{otherwise}.
\end{cases}$$ 
In each case, a right eigenvector $\W$ and a left eigenvector $\Z$ of $B$ corresponding to $\lambda$ are given by:
$$w_i=\begin{cases}
u_1,&\mbox{if }1\leq i\leq k-1;\\
u_{i-k+1},&\mbox{if }k\leq i\leq k+n-1.
\end{cases} \qquad \qquad
z_i=\begin{cases}
\frac{\epsilon a_{11}v_1}{\lambda},&\mbox{if }1\leq i\leq s-1;\\
v_1,&\mbox{if }s\leq i\leq k-1;\\
v_{i-k+1},&\mbox{if }k\leq i\leq k+n-1.
\end{cases}$$
Consequently, $B$ is algebraically positive. Further, for some suitable $\epsilon$ satisfying $0<\epsilon<1$,
$$b_{ks}z_k+\sum\limits_{i=1}^{n-1}b_{k+i,s}z_{k+i}=(1-\epsilon)a_{11} v_1+\sum\limits_{i=1}^{n-1}a_{i+1,1}v_{i+1}=(1-\epsilon)a_{11} v_1+(\lambda-a_{11})v_1=\left(\lambda-\epsilon a_{11}\right) v_1>0.$$
\end{enumerate}
\end{theorem}

\begin{remark}
Introducing the row and the column indices, we can write the matrices in Theorem \ref{thm2.4} as follows.
\begin{enumerate}
\item
$$\begin{blockarray}{ccccccccc}
\begin{block}{c[ccccc|ccc]}
_1&0&\lambda&0&\cdots&0&0&\cdots&0 \\
\vdots&\vdots&\ddots&\ddots&\ddots&\vdots&\vdots&&\vdots \\
\vdots&\vdots&&\ddots&\ddots&0&\vdots&&\vdots \\
\vdots&0&\cdots&\cdots&0&\lambda&0&\cdots&0 \\
_k&a_{11}&0&\cdots&\cdots&0&a_{12}&\cdots&a_{1n} \\
\cline{2-9}
_{k+1}&0&\cdots&a_{21}&\cdots&0&a_{22}&\cdots&a_{2n} \\
\vdots&\vdots&&\vdots&&\vdots&\vdots&&\vdots \\
_{k+n-1}&0&\cdots&a_{n1}&\cdots&0&a_{n2}&\cdots&a_{nn} \\
\end{block}\\[-4 mm]
&_1&\ldots&_j&\ldots&_k&_{k+1}&\ldots&_{k+n-1} \\
\end{blockarray}$$
\item If $s\leq j$, then we have $B$ as follows.
$$\begin{blockarray}{ccccccccccccc}
\begin{block}{c[ccccccccc|ccc]}
_1&0&\lambda&0&\cdots&\cdots&\cdots&\cdots&\cdots&0&0&\cdots&0 \\
\vdots&\vdots&\ddots&\ddots&\ddots&&&&&\vdots&\vdots&&\vdots \\
\vdots&\vdots&&\ddots&\ddots&\ddots&&&&\vdots&\vdots&&\vdots \\
_{j-1}&0&\cdots&\cdots&0&\lambda&\ddots&&&\vdots&\vdots&&\vdots \\
_j&0&\cdots&(1-\epsilon)\lambda&\cdots&0&\epsilon\lambda&\ddots&&\vdots&\vdots&&\vdots \\
_{j+1}&0&\cdots&\cdots&\cdots&\cdots&0&\lambda&\ddots&\vdots&\vdots& &\vdots \\
\vdots&\vdots&&&&&&\ddots&\ddots&0&\vdots&&\vdots \\
\vdots&0&\cdots&\cdots&\cdots&\cdots&\cdots&\cdots&0&\lambda&0&\cdots&0 \\
_k&a_{11}&0&\cdots&\cdots&\cdots&\cdots&\cdots&\cdots&0&a_{12}&\cdots&a_{1n} \\
\cline{2-13}
_{k+1}&0&\cdots&a_{21}&\cdots&\cdots&\cdots&\cdots&\cdots&0&a_{22}&\cdots&a_{2n} \\
\vdots&\vdots&&\vdots&&&&&&\vdots&\vdots&&\vdots \\
_{k+n-1}&0&\cdots&a_{n1}&\cdots&\cdots&\cdots&\cdots&\cdots&0&a_{n2}&\cdots&a_{nn} \\
\end{block}\\[-4 mm]
&_1&\ldots&_s&\ldots&\ldots&_{j+1}&\ldots&\ldots&_k&_{k+1}&\cdots&_{k+n-1} \\
\end{blockarray}$$
If $s\geq j+2$, then we have $B$ as follows.
$$\begin{blockarray}{cccccccccccccc}
\begin{block}{c[cccccccccc|ccc]}
_1&0&\lambda&0&\cdots&\cdots&\cdots&\cdots&\cdots&\cdots&0&0&\cdots&0 \\
\vdots&\vdots&\ddots&\ddots&\ddots&&&&&&\vdots&\vdots&&\vdots \\
_{j-1}&\vdots&&\ddots&\lambda&0&\cdots&\cdots&\cdots&\cdots&0& \vdots&&\vdots \\
_j&\vdots&&&\ddots&\epsilon\lambda&0&\cdots&(1-\epsilon)\lambda&\cdots&0&\vdots&&\vdots \\
_{j+1}&\vdots&&&&\ddots&\lambda&0&\cdots&\cdots&0&\vdots& &\vdots \\
\vdots&\vdots&&&&&\ddots&\ddots&\ddots&&\vdots&\vdots&&\vdots \\
\vdots&\vdots&&&&&&\ddots&\ddots&\ddots&\vdots&\vdots&&\vdots \\
\vdots&\vdots&&&&&&&\ddots&\ddots&0&\vdots&&\vdots \\
\vdots&0&\cdots&\cdots&\cdots&\cdots&\cdots&\cdots&\cdots&0&\lambda&0&\cdots&0 \\
_k&a_{11}&0&\cdots&\cdots&\cdots&\cdots&\cdots&\cdots&\cdots&0&a_{12}&\cdots&a_{1n} \\
\cline{2-14}
_{k+1}&0&\cdots&\cdots&\cdots&\cdots&\cdots&\cdots&a_{21}&\cdots&0&a_{22}&\cdots&a_{2n} \\
\vdots&\vdots&&&&&&&\vdots&&\vdots&\vdots&&\vdots \\
_{k+n-1}&0&\cdots&\cdots&\cdots&\cdots&\cdots&\cdots&a_{n1}&\cdots&0&a_{n2}&\cdots&a_{nn} \\
\end{block}\\[-3 mm]
&_1&\ldots&\ldots&\ldots&_{j+1}&\ldots&\ldots&_s&\ldots&_k&_{k+1}&\ldots&_{k+n-1} \\
\end{blockarray}$$
\item
$$\begin{blockarray}{cccccccccc}
\begin{block}{c[cccccc|ccc]}
_1&0&\lambda&0&\cdots&\cdots&0&0&\cdots&0 \\
\vdots&\vdots&\ddots&\ddots&\ddots&&\vdots&\vdots&&\vdots \\
\vdots&\vdots&&\ddots&\ddots&\ddots&\vdots&\vdots&&\vdots \\
\vdots&\vdots&&&\ddots&\ddots&0&\vdots&&\vdots \\
\vdots&0&\cdots&\cdots&\cdots&0&\lambda&0&\cdots&0 \\
_k&\epsilon a_{11}&0&\cdots&(1-\epsilon)a_{11}&\cdots&0&a_{12}&\cdots&a_{1n} \\
\cline{2-10}
_{k+1}&0&\cdots&\cdots&a_{21}&\cdots&0&a_{22}&\cdots&a_{2n} \\
\vdots&\vdots&&&\vdots&&\vdots&\vdots&&\vdots \\
_{k+n-1}&0&\cdots&\cdots&a_{n1}&\cdots&0&a_{n2}&\cdots&a_{nn} \\
\end{block}\\[-3 mm]
&_1&\ldots&\ldots&_s&\ldots&_k&_{k+1}&\ldots&_{k+n-1} \\
\end{blockarray}$$
\end{enumerate}
\end{remark}

\begin{theorem}\label{thm2.5}
Let $0<\epsilon<1$, and let $A$ be a real square matrix of order $n$ such that $a_{1j}>0$ for some $j\in\{1,2,\ldots,n\}$. Suppose that $\lambda>0$ is a simple eigenvalue of $A$ such that $A\U=\lambda\U$ and $\V^TA=\lambda \V^T$ for some $\U,\V>\0$. If $a_{1j}v_1+\sum\limits_{i=1}^{n-k}a_{k+i,j}v_{k+i}>0$ for some  $k\in\{1,2,\ldots,n\}$ $($for $k=n$, the hypothesis is $a_{1j}>0)$, then the following statements are true. 
\begin{enumerate}
\item For every $s\in\{1,2,\ldots,m\}$, $\lambda$ is a simple eigenvalue of $B$, where $B$ is of order $m+n$, given by
$$b_{pq}=\begin{cases}
\lambda,&\mbox{if }(p,q)\in\{(m,m+j)\}\cup\{(1,2),(2,3),\ldots,(m-1,m)\};\\
a_{1j},&\mbox{if }(p,q)=(m+1,1);\\
a_{p-m,j},&\mbox{if }p\in\{m+k+1,m+k+2,\ldots,m+n\}\mbox{ and }q=s;\\
a_{p-m,q-m},&\mbox{if }p,q\in\{m+1,m+2,\ldots,m+n\}\mbox{ such that }\\
&\hspace*{1 cm}(p,q)\notin\{(m+1,m+j)\cup\{(m+k+1,m+j),\ldots,(m+n,m+j)\}\};\\
0,&\mbox{otherwise}.
\end{cases}$$  
Moreover, a right eigenvector $\W$ and a left eigenvector $\Z$ of $B$ corresponding to $\lambda$ are given by:
$$w_i=\begin{cases}
u_j,&\mbox{if }1\leq i\leq m;\\
u_{i-m},&\mbox{if }m+1\leq i\leq m+n.
\end{cases}  \quad
z_i=\begin{cases}
\frac{a_{1j}v_1}{\lambda},&\mbox{if }1\leq i\leq s-1;\\
\frac{1}{\lambda}\left(a_{1j}v_1+\sum\limits_{i=1}^{n-k}a_{k+i,j}v_{k+i}\right),&\mbox{if }s\leq i\leq m;\\
v_{i-m},&\mbox{if }m+1\leq i\leq m+n.
\end{cases}$$
Consequently, $B$ is algebraically positive.
\item For every $s\in\{1,2,\ldots,m-1\}$ and $t\in\{1,2,\ldots,m\}\setminus\{s+1\}$, $\lambda$ is a simple eigenvalue of $B$, where $B$ is of order $m+n$, given by
$$b_{pq}=\begin{cases}
\lambda,&\mbox{if }(p,q)\in(\{(m,m+j)\}\cup\{(1,2),(2,3),\ldots,(m-1,m)\})\setminus\{(s,s+1)\};\\
\epsilon\lambda,&\mbox{if }(p,q)=(s,s+1);\\
a_{1j},&\mbox{if }(p,q)=(m+1,1);\\
(1-\epsilon)\lambda,&\mbox{if }(p,q)=(s,t);\\
a_{p-m,j},&\mbox{if }p\in\{m+k+1,m+k+2,\ldots,m+n\}\mbox{ and }q=t;\\
a_{p-m,q-m},&\mbox{if }p,q\in\{m+1,m+2,\ldots,m+n\}\mbox{ such that }\\
&\hspace*{1 cm}(p,q)\notin\{(m+1,m+j)\}\cup\{(m+k+1,m+j),\ldots,(m+n,m+j)\}\};\\
0,&\mbox{otherwise}.
\end{cases}$$  
If $t\leq s$, then a right eigenvector $\W$ and a left eigenvector $\Z$ of $B$ corresponding to $\lambda$ are given by:
$$w_i=\begin{cases}
u_j,&\mbox{if }1\leq i\leq m;\\
u_{i-m},&\mbox{if }m+1\leq i\leq m+n.
\end{cases} \quad 
z_i=\begin{cases}
\frac{a_{1j}v_1}{\lambda},&\mbox{if }1\leq i\leq t-1;\\
\frac{1}{\epsilon\lambda}\left(a_{1j}v_1+\sum\limits_{i=1}^{n-k}a_{k+i,j}v_{k+i}\right),&\mbox{if }t\leq i\leq s;\\
\frac{1}{\lambda}\left(a_{1j}v_1+\sum\limits_{i=1}^{n-k}a_{k+i,j}v_{k+i}\right),&\mbox{if }s+1\leq i\leq m;\\
v_{i-m},&\mbox{if }m+1\leq i\leq m+n.
\end{cases}$$
Consequently, in each case, $B$ is algebraically positive. Further, for some suitable $\epsilon$ satisfying $0<\epsilon<1$,
\begin{align*}
b_{st}z_s+\sum\limits_{i=1}^{n-k}b_{m+k+i,t}z_{m+k+i}&=\frac{1-\epsilon}{\epsilon}\left(a_{1j}v_1+\sum\limits_{i=1}^{n-k}a_{k+i,j}v_{k+i}\right)+\sum\limits_{i=1}^{n-k}a_{k+i,j}v_{k+i}\\&=\frac{1}{\epsilon}\left(a_{1j}v_1+\sum\limits_{i=1}^{n-k}a_{k+i,j}v_{k+i}\right)-a_{1j}v_1>0.
\end{align*}
If $t\geq s+2$, then a right eigenvector $\W$ and a left eigenvector $\Z$ of $B$ corresponding to $\lambda$ are given by:
$$w_i=\begin{cases}
u_j,&\mbox{if }1\leq i\leq m;\\
u_{i-m},&\mbox{if }m+1\leq i\leq m+n.
\end{cases} \quad
z_i=\begin{cases}
\frac{a_{1j}v_1}{\lambda},&\mbox{if }1\leq i\leq s;\\
\frac{\epsilon a_{1j}v_1}{\lambda},&\mbox{if }s+1\leq i\leq t-1;\\
\frac{1}{\lambda}\left(a_{1j}v_1+\sum\limits_{i=1}^{n-k}a_{k+i,j}v_{k+i}\right),&\mbox{if }t\leq i\leq m;\\
v_{i-m},&\mbox{if }m+1\leq i\leq m+n.
\end{cases}$$
Consequently, in each case, $B$ is algebraically positive. Further, for some suitable $\epsilon$ satisfying $0<\epsilon<1$,
$$b_{st}z_s+\sum\limits_{i=1}^{n-k}b_{m+k+i,t}z_{m+k+i}=(1-\epsilon)a_{1j}v_1+\sum\limits_{i=1}^{n-k}a_{k+i,j}v_{k+i}=\left(a_{1j}v_1+\sum\limits_{i=1}^{n-k}a_{k+i,j}v_{k+i}\right)-\epsilon a_{1j}v_1>0.$$
\item For every $t\in\{1,2,\ldots,m\}$, $\lambda$ is a simple eigenvalue of $B$, where $B$ is of order $m+n$, given by
$$b_{pq}=\begin{cases}
\lambda,&\mbox{if }(p,q)\in\{(1,2),(2,3),\ldots,(m-1,m)\};\\
\epsilon\lambda,&\mbox{if }(p,q)=(m,m+j);\\
a_{1j},&\mbox{if }(p,q)=(m+1,1);\\
(1-\epsilon)\lambda,&\mbox{if }(p,q)=(m,t);\\
a_{p-m,j},&\mbox{if }p\in\{m+k+1,m+k+2,\ldots,m+n\}\mbox{ and }q=t;\\
a_{p-m,q-m},&\mbox{if }p,q\in\{m+1,m+2,\ldots,m+n\}\mbox{ such that }\\
&\hspace*{1 cm}(p,q)\notin\{(m+1,m+j)\cup\{(m+k+1,m+j),\ldots,(m+n,m+j)\}\};\\
0,&\mbox{otherwise}.
\end{cases}$$  
Moreover, a right eigenvector $\W$ and a left eigenvector $\Z$ of $B$ corresponding to $\lambda$ are given by:
$$w_i=\begin{cases}
u_j,&\mbox{if }1\leq i\leq m;\\
u_{i-m},&\mbox{if }m+1\leq i\leq m+n.
\end{cases} \quad 
z_i=\begin{cases}
\frac{a_{1j}v_1}{\lambda},&\mbox{if }1\leq i\leq t-1;\\
\frac{1}{\epsilon\lambda}\left(a_{1j}v_1+\sum\limits_{i=1}^{n-k}a_{k+i,j}v_{k+i}\right),&\mbox{if }t\leq i\leq m;\\
v_{i-m},&\mbox{if }m+1\leq i\leq m+n.
\end{cases}$$
Consequently, in each case, $B$ is algebraically positive. Further, for some suitable $\epsilon$ satisfying $0<\epsilon<1$,
\begin{align*}
b_{mt}z_m+\sum\limits_{i=1}^{n-k}b_{m+k+i,t}z_{m+k+i}&=\frac{1-\epsilon}{\epsilon}\left(a_{1j}v_1+\sum\limits_{i=1}^{n-k}a_{k+i,j}v_{k+i}\right)+\sum\limits_{i=1}^{n-k}a_{k+i,j}v_{k+i}\\&=\frac{1}{\epsilon}\left(a_{1j}v_1+\sum\limits_{i=1}^{n-k}a_{k+i,j}v_{k+i}\right)-a_{1j}v_1>0.
\end{align*}
\end{enumerate}
\end{theorem}

\begin{remark}
Introducing the row and the column indices, we can write the matrices in Theorem \ref{thm2.5} as follows.
\begin{enumerate}
\item 
$$\begin{blockarray}{ccccccccccccc}
\begin{block}{c[ccccc|ccccccc]}
_1& 0&\lambda&0&\cdots&0& 0&\cdots&\cdots&\cdots&\cdots&\cdots&0 \\
\vdots& \vdots&\ddots&\ddots&\ddots&\vdots& \vdots&&&&&&\vdots \\
\vdots& \vdots&&\ddots&\ddots&0& \vdots&&&&&&\vdots \\
\vdots& \vdots&&&\ddots&\lambda& 0&\cdots&\cdots&\cdots&\cdots&\cdots&0 \\
_m& 0&\cdots&\cdots&\cdots&0& 0&\cdots&0&\lambda&0&\cdots&0 \\
\cline{2-13}
_{m+1}& a_{1j}&0&\cdots&\cdots&0& a_{11}&\cdots&a_{1,j-1}&0&a_{1,j+1}&\cdots&a_{1n} \\
_{m+2}& 0&\cdots&\cdots&\cdots&0& a_{21}&\cdots&\cdots&a_{2j}&\cdots&\cdots&a_{2n} \\
\vdots& \vdots&&&&\vdots& \vdots&&&&&&\vdots \\
_{m+k}& 0&\cdots&\cdots&\cdots&0& a_{k1}&\cdots&\cdots&a_{kj}&\cdots&\cdots&a_{kn} \\
\cline{2-13}
_{m+k+1}& 0&\cdots&a_{k+1,j}&\cdots&0& a_{k+1,1}&\cdots&a_{k+1,j-1}&0&a_{k+1,j+1}&\cdots&a_{k+1,n} \\
\vdots& \vdots&&\vdots&&\vdots& \vdots&&\vdots&\vdots&\vdots&&\vdots \\
_{m+n}& 0&\cdots&a_{nj}&\cdots&0& a_{n1}&\cdots&a_{n,j-1}&0&a_{n,j+1}&\cdots&a_{nn} \\
\end{block}\\[-3 mm]
&_1&\ldots&_s&\ldots&_m&_{m+1}&\ldots&\ldots&_{m+j}&\ldots&\ldots&_{m+n} \\
\end{blockarray}$$
\item If $t\leq s$, then we have $B$ as follows.
$$\makeatletter\setlength\BA@colsep{2pt}\makeatother
\begin{blockarray}{ccccccccccccccccc}
\begin{block}{c[ccccccccc|ccccccc]}
_1&0&\lambda&0&\cdots&\cdots&\cdots&\cdots&\cdots&0&0&\cdots& \cdots&\cdots&\cdots&\cdots&0 \\
\vdots&\vdots&\ddots&\ddots&\ddots&&&&&\vdots&\vdots&&&&&&\vdots \\
\vdots&\vdots&&\ddots&\ddots&\ddots&&&&\vdots&\vdots&&&&&&\vdots \\
_{s-1}&0&\cdots&\cdots&0&\lambda&\ddots&&&\vdots&\vdots&&&&&&\vdots \\
_s&0&\cdots&(1-\epsilon)\lambda&\cdots&0&\epsilon\lambda&\ddots&&\vdots&\vdots&&&&&&\vdots \\
_{s+1}&0&\cdots&\cdots&\cdots&\cdots&0&\lambda&\ddots&\vdots& \vdots&&&&& &\vdots \\
\vdots&\vdots&&&&&&\ddots&\ddots&0&\vdots&&&&&&\vdots \\
\vdots&0&\cdots&\cdots&\cdots&\cdots&\cdots&\cdots&0&\lambda&0& \cdots&\cdots&\cdots&\cdots&\cdots&0 \\
_m& 0&\cdots&\cdots&\cdots&\cdots&\cdots&\cdots&\cdots&0& 0&\cdots&\cdots&\lambda&\cdots&\cdots&0 \\
\cline{2-17}
_{m+1}& a_{1j}&0&\cdots&\cdots&\cdots&\cdots&\cdots&\cdots&0& a_{11}&\cdots&a_{1,j-1}&0&a_{1,j+1}&\cdots&a_{1n} \\
_{m+2}& 0&\cdots&\cdots&\cdots&\cdots&\cdots&\cdots&\cdots&0& a_{21}&\cdots&\cdots&a_{2j}&\cdots&\cdots&a_{2n} \\
\vdots& \vdots&&&&&&&&\vdots& \vdots&&&&&&\vdots \\
_{m+k}& 0&\cdots&\cdots&\cdots&\cdots&\cdots&\cdots&\cdots&0& a_{k1}&\cdots&\cdots&a_{kj}&\cdots&\cdots&a_{kn} \\
\cline{2-17}
_{m+k+1}& 0&\cdots&a_{k+1,j}&\cdots&\cdots&\cdots&\cdots&\cdots&0& a_{k+1,1}&\cdots&a_{k+1,j-1}&0&a_{k+1,j+1}&\cdots&a_{k+1,n} \\
\vdots& \vdots&&\vdots&&&&&&\vdots& \vdots&&\vdots&\vdots&\vdots&&\vdots \\
_{m+n}& 0&\cdots&a_{nj}&\cdots&\cdots&\cdots&\cdots&\cdots&0& a_{n1}&\cdots&a_{n,j-1}&0&a_{n,j+1}&\cdots&a_{nn} \\
\end{block}\\[-3 mm]
&_1&\ldots&_t&\ldots&\ldots&_{s+1}&\cdots&\ldots&_m&_{m+1}&\ldots&\ldots&_{m+j}& \ldots&\ldots&_{m+n} \\
\end{blockarray}$$
If $t\geq s+2$, then we have $B$ as follows.
$$\makeatletter\setlength\BA@colsep{2pt}\makeatother
\begin{blockarray}{cccccccccccccccccc}
\begin{block}{c[cccccccccc|ccccccc]}
_1&0&\lambda&0&\cdots&\cdots&\cdots&\cdots&\cdots&\cdots&0&0&\cdots&\cdots&\cdots&\cdots&\cdots&0 \\
\vdots&\vdots&\ddots&\ddots&\ddots&&&&&&\vdots&\vdots&&&&&&\vdots \\
_{s-1}&\vdots&&\ddots&\lambda&0&\cdots&\cdots&\cdots&\cdots&0&\vdots& &&&&&\vdots \\
_s&\vdots&&&\ddots&\epsilon\lambda&0&\cdots&(1-\epsilon)\lambda&\cdots&0&\vdots&&&&&&\vdots \\
_{s+1}&\vdots&&&&\ddots&\lambda&0&\cdots&\cdots&0&\vdots&&&&&&\vdots \\
\vdots&\vdots&&&&&\ddots&\ddots&\ddots&&\vdots&\vdots&&&&&&\vdots \\
\vdots&\vdots&&&&&&\ddots&\ddots&\ddots&\vdots&\vdots&&&&&&\vdots \\
\vdots&\vdots&&&&&&&\ddots&\ddots&0&\vdots&&&&&&\vdots \\
\vdots&\vdots&&&&&&&&\ddots&\lambda&0 &\cdots&\cdots&\cdots&\cdots&\cdots&0 \\
_m&0&\cdots&\cdots&\cdots&\cdots&\cdots&\cdots&\cdots&\cdots&0&0& \cdots&\cdots&\lambda&\cdots&\cdots&0 \\
\cline{2-18}
_{m+1}& a_{1j}&0&\cdots&\cdots&\cdots&\cdots&\cdots&\cdots&\cdots&0& a_{11}&\cdots&a_{1,j-1}&0&a_{1,j+1}&\cdots&a_{1n} \\
_{m+2}& 0&\cdots&\cdots&\cdots&\cdots&\cdots&\cdots&\cdots&\cdots&0& a_{21}&\cdots&\cdots&a_{2j}&\cdots&\cdots&a_{2n} \\
\vdots& \vdots&&&&&&&&&\vdots& \vdots&&&&&&\vdots \\
_{m+k}& 0&\cdots&\cdots&\cdots&\cdots&\cdots&\cdots&\cdots&\cdots&0& a_{k1}&\cdots&\cdots&a_{kj}&\cdots&\cdots&a_{kn} \\
\cline{2-18}
_{m+k+1}& 0&\cdots&\cdots&\cdots&\cdots&\cdots&\cdots&a_{k+1,j}&\cdots&0& a_{k+1,1}&\cdots&a_{k+1,j-1}&0&a_{k+1,j+1}&\cdots&a_{k+1,n} \\
\vdots& \vdots&&&&&&&\vdots&&\vdots& \vdots&&\vdots&\vdots&\vdots&&\vdots \\
_{m+n}& 0&\cdots&\cdots&\cdots&\cdots&\cdots&\cdots&a_{nj}&\cdots&0& a_{n1}&\cdots&a_{n,j-1}&0&a_{n,j+1}&\cdots&a_{nn} \\
\end{block}\\[-3 mm]
&_1&\ldots&\ldots&\ldots&_{s+1}&\ldots&\ldots&_t&\ldots&_m&_{m+1}&\ldots &\ldots& _{m+j}&\ldots&\ldots&_{m+n} \\
\end{blockarray}$$
\item
$$\begin{blockarray}{ccccccccccccc}
\begin{block}{c[ccccc|ccccccc]}
_1&0&\lambda&0&\cdots&0&0&\cdots&\cdots&\cdots&\cdots&\cdots&0 \\
\vdots&\vdots&\ddots&\ddots&&\vdots&\vdots&&&&&&\vdots \\
\vdots&\vdots&\ddots&\ddots&\ddots&\vdots&\vdots&&&&&&\vdots \\
\vdots&\vdots&&\ddots&\ddots&0&\vdots&&&&&&\vdots \\
\vdots&0&\cdots&\cdots&0&\lambda&0&\cdots&\cdots&\cdots&\cdots &\cdots&0 \\
_m&0&\cdots&(1-\epsilon)\lambda &\cdots&0& 0&\cdots&\cdots&\epsilon\lambda&\cdots&\cdots&0 \\
\cline{2-13}
_{m+1}& a_{1j}&0&\cdots&\cdots&0& a_{11}&\cdots&a_{1,j-1}&0&a_{1,j+1}&\cdots&a_{1n} \\
_{m+2}& 0&\cdots&\cdots&\cdots&0& a_{21}&\cdots&\cdots&a_{2j}&\cdots&\cdots&a_{2n} \\
\vdots& \vdots&&&&\vdots& \vdots&&&&&&\vdots \\
_{m+k}& 0&\cdots&\cdots&\cdots&0& a_{k1}&\cdots&\cdots&a_{kj}&\cdots&\cdots&a_{kn} \\
\cline{2-13}
_{m+k+1}& 0&\cdots&a_{k+1,j}&\cdots&0& a_{k+1,1}&\cdots&a_{k+1,j-1}&0&a_{k+1,j+1}&\cdots&a_{k+1,n} \\
\vdots& \vdots&&\vdots&&\vdots& \vdots&&\vdots&\vdots&\vdots&&\vdots \\
_{m+n}& 0&\cdots&a_{nj}&\cdots&0& a_{n1}&\cdots&a_{n,j-1}&0&a_{n,j+1}&\cdots&a_{nn} \\
\end{block}\\[-3 mm]
&_1&\ldots&_t&\ldots&_m&_{m+1}&\ldots &\ldots& _{m+j}&\ldots&\ldots&_{m+n} \\
\end{blockarray}$$
\end{enumerate}
\end{remark}

\section{A sufficient condition for a sign pattern matrix to allow algebraic positivity}\label{sec3}
Let $A$ be a sign pattern matrix of order $n$. Let the matrices $A_+, A_-$ and $B_A$ be defined as follows:
$$(A_+)_{ij}=\begin{cases}
+, & \mbox{if }a_{ij}=+;\\
0, & \mbox{if }a_{ij}\neq+.
\end{cases} \qquad
(A_-)_{ij}=\begin{cases}
-, & \mbox{if }a_{ij}=-;\\
0, & \mbox{if }a_{ij}\neq-.
\end{cases}  \qquad
B_A=A_+-(A_-)^T.$$
We have the following necessary conditions on algebraic positivity associated with the allow problem.

\begin{theorem}[\cite{KQZ16}]\label{thm3.1}
Every algebraically positive matrix is irreducible.
\end{theorem}

\begin{theorem}[\cite{KQZ16}]\label{thm3.2}
If a sign pattern matrix $A$ allows algebraic positivity, then either every row and column of $A$ contains a $+$, or every row and column of $A$ contains a $-$.
\end{theorem}

\begin{theorem}[\cite{AP19}]\label{thm3.3}
If a sign pattern matrix $A$ allows algebraic positivity, then $B_A$ is irreducible.
\end{theorem}

We say that an irreducible sign pattern matrix $A$ is AP-irreducible if every row and column of $A$ contains a $+$ and $B_A$ is irreducible. We say that a sign pattern matrix $A$ is minimally AP-irreducible if it does not remain AP-irreducible after replacement of a nonzero entry with 0. 

Theorem \ref{thm3.1}, \ref{thm3.2} and \ref{thm3.3} together imply that if $A$ is algebraically positive, then either $A$ or $-A$ is AP-irreducible. Whether the converse is true? We do not know the answer yet, but we give some results which indicate that the converse may be true.

We will present a class of AP-irreducible sign pattern matrices through Theorem \ref{thm3.16} that allow algebraic positivity. On that purpose, we give the following results as prerequisites.

\begin{lemma}[\cite{KQZ16}]\label{lem3.5}
If $A$ is algebraically positive, then for every permutation matrix $P$, $P^TAP$ is algebraically positive.
\end{lemma}

If a sign pattern matrix requires some property, then it also allows that property. So we have the following result from \cite[Lemma 6]{DB19}.

\begin{lemma}[\cite{DB19}]\label{lem3.6}
If all nonzero off-diagonal entries of an irreducible sign pattern matrix are the same, which is $+$ or $-$, then that sign pattern matrix allows algebraic positivity.
\end{lemma}

\begin{theorem}\label{thm3.7}
If a sign pattern matrix $A$ allows algebraic positivity, then every super-pattern of $A$ allows algebraic positivity.
\end{theorem}

\begin{proof}
Let $A$ be a sign pattern matrix and $\tilde{A}\in Q(A)$ be algebraically positive. Then $p(\tilde{A})$ is a positive matrix for some real polynomial $p$.

Let $X$ be a super-pattern of $A$, and $B\in Q(X)$ be defined as follows:
$$b_{ij}=\begin{cases}
1,&\mbox{if }x_{ij}=+; \\
-1,&\mbox{if }x_{ij}=-; \\
0,&\mbox{if }x_{ij}=0.
\end{cases}$$
For every $\epsilon>0$, $p(\tilde{A}+\epsilon B)=p(\tilde{A})+\epsilon\cdot q(\tilde{A},B,\epsilon)$, where $q(\tilde{A},B,\epsilon)$ is a polynomial in $\tilde{A},B$ and $\epsilon$. Since $p(\tilde{A})$ is a positive matrix, there exists $\epsilon>0$ such that $p(\tilde{A}+\epsilon B)$ is also a positive matrix. Since $X$ is a super-pattern of $A$, $\tilde{A}+\epsilon B\in Q(X)$. Therefore $X$ allows algebraic positivity.
\end{proof}

\begin{lemma}\label{lem3.8}
Let $D$ be a minimally strongly connected digraph. For every $v\in V(D)$, there exists a nested sequence $V_1,V_2,\ldots,V_k$ of vertex sets such that $v\in V_1\subsetneq V_2\subsetneq \cdots\subsetneq V_k=V(D)$ and $V_1,V_2,\ldots,V_k$ satisfies the following properties:
\begin{enumerate}
\item[$\mathrm{P1}$.] The subgraph of $D$ induced by $V_1$ is a directed cycle.
\item[$\mathrm{P2}$.] For every $i\in\{2,3,\ldots,k\}$, the subgraph of $D$ induced by $V_i$ is strongly connected. 
\item[$\mathrm{P3}$.] For every $i\in\{2,3,\ldots,k\}$, there is no strongly connected subgraph of $D$ induced by some $V$ satisfying $V_{i-1}\subsetneq V\subsetneq V_i$.
\end{enumerate}
\end{lemma} 

\begin{proof}
We prove the result using the following constructive method.
\begin{enumerate}
\item[Step 1.] Let $v\in V(D)$. Since $D$ is strongly connected, there exists a directed cycle containing $v$. Denote the vertex set of that directed cycle by $V_1$. Since $D$ is minimally strongly connected, the subgraph of $D$ induced by $V_1$ is that directed cycle.
\item[Step 2.] Suppose that $V_1,V_2,\ldots,V_{i-1}$ are obtained as required. If $V_{i-1}=V(D)$, then we are done. Otherwise, go to Step 3.
\item[Step 3.] If $V_{i-1}\subsetneq V(D)$, then there exists $x\in V_{i-1}$ such that $(\overrightarrow{x,y})$ is an arc of $D$ for some $y\in V(D)\setminus V_{i-1}$. Since $D$ is strongly connected, consider a directed path in $D$ that starts from $y$ and ends at some vertex of $V_{i-1}$, say $z$, containing no vertex from $V_{i-1}$ except $z$. Let $V(y\rightarrow z)$ be the set of vertices on that directed path. If $V_i=V_{i-1}\cup V(y\rightarrow z)$, then the subgraph of $D$ induced by $V_i$ is strongly connected. Since $D$ is minimally strongly connected, there is no strongly connected subgraph of $D$ induced by some $V$ satisfying $V_{i-1}\subsetneq V\subsetneq V_i$. Now go to Step 2 with $V_i$.
\end{enumerate}
Since $V(D)$ is finite, we get a positive integer $k$ such that $V_k=V(D)$.
\end{proof}

\begin{remark}\label{rem3.9}
For every $i\in\{2,3,\ldots,k\}$, there exist precisely two arcs such that one of them has the initial vertex (say, $p_{i-1}$) in $V_{i-1}$ and the terminal vertex in $V_i\setminus V_{i-1}$, and the other arc has the initial vertex in $V_i\setminus V_{i-1}$ and the terminal vertex (say, $q_{i-1}$) in $V_{i-1}$. If $V_i\setminus V_{i-1}=\{v_{i,1},v_{i,2},\ldots,v_{i,n_i}\}$, then without loss of generality, we may assume that $\{(\overrightarrow{p_{i-1},v_{i,1}}),(\overrightarrow{v_{i,n_i},q_{i-1}})\}\cup\{(\overrightarrow{v_{i,1},v_{i,2}}),\ldots,(\overrightarrow{v_{i,n_i-1},v_{i,n_i}})\}$ is the set of arcs of the subgraph of $D$ induced by $V_i$ such that none of them is an arc of the subgraph of $D$ induced by $V_{i-1}$.
\end{remark}

Let us recall that a matrix is minimally irreducible if it does not remain irreducible after replacement of a nonzero entry with 0.

\begin{lemma}\label{lem3.10}
A minimally irreducible matrix of order $n$ can have at most $2n-2$ nonzero entries.
\end{lemma}

\begin{proof}
Let $A$ be a minimally irreducible matrix of order $n$. Then $D(A)$ is minimally strongly connected. So there exists a nested sequence $V_1,V_2,\ldots,V_k$ of vertex sets satisfying the properties P1, P2, P3 of Lemma \ref{lem3.8} such that $V_1\subsetneq V_2\subsetneq \cdots\subsetneq V_k=V(D(A))$. Let the cardinalities of $V_1,V_2\setminus V_1,\ldots,V_k\setminus V_{k-1}$ be $n_1,n_2,\ldots,n_k$, respectively. Then $n_1+n_2+\cdots+n_k=n$. Since $D(A)$ is minimally strongly connected, it has no loop, and thus $n_1\geq 2$. Therefore $k\leq n-1$. Further, using Remark \ref{rem3.9}, we can conclude that the subgraph of $D(A)$ induced by $V_i$ has precisely $n_i+1$ more arcs than the subgraph of $D(A)$ induced by $V_{i-1}$ for $i=2,3,\ldots,k$. Therefore the total number of arcs in $D(A)$ is $n_1+n_2+\cdots+n_k+k-1\leq 2n-2$. The number of nonzero entries of $A$ is the number of arcs in $D(A)$. Therefore the number of nonzero entries of $A$ is at most $2n-2$.
\end{proof}

For any two subsets $\alpha,\beta$ of $\{1,2,\ldots,n\}$, let $A[\alpha,\beta]$ is the submatrix of $A$ with rows and columns corresponding to the indices in $\alpha$ and $\beta$ respectively. When $\alpha=\beta$, then we write $A[\alpha]$ instead of $A[\alpha,\alpha]$.

Let $A$ be a sign pattern matrix of order $n$. There exist $m$ nonempty and pairwise disjoint subsets of $\{1,2,\ldots,n\}$, say $\alpha_1,\alpha_2,\ldots,\alpha_m$ such that $\alpha_1\cup\alpha_2\cup\cdots\cup\alpha_m=\{1,2,\ldots,n\}$, and for every $i\in\{1,2,\ldots,m\}$, $A[\alpha_i]$ is an irreducible matrix and there exists no $\beta_i$ satisfying $\alpha_i\subsetneq\beta_i\subseteq\{1,2,\ldots,n\}$ such that $A[\beta_i]$ is irreducible. Each such $\alpha_i$ is said to be an irreducible component of $A$.

We have the following result about AP-irreducible sign pattern matrices.

\begin{lemma}\label{lem3.11}
Suppose that $A$ is a minimally AP-irreducible sign pattern matrix of order $n$ with all diagonal entries equal to $0$. If $A[\alpha,\beta]$ contains no $+$ for any two distinct irreducible components $\alpha$ and $\beta$ of $A_+$, then $A$ has at most $2n-2$ nonzero entries. Consequently, at least two rows of $A$ and at least two columns of $A$ contain precisely one nonzero entry.
\end{lemma}

\begin{proof}
Let $\alpha_1,\alpha_2,\ldots,\alpha_m$ be the distinct irreducible components of $A_+$ with cardinalities $n_1,n_2,\ldots,n_m$, respectively. Then $\alpha_1\cup\cdots\cup\alpha_m=\{1,2,\ldots,n\}$ and $n_1+n_2+\cdots+n_m=n$. Since $A$ is minimally AP-irreducible, $A[\alpha_i]=A_+[\alpha_i]$ is minimally irreducible for $i=1,2,\ldots,m$. Moreover, for any two distinct irreducible components $\alpha$ and $\beta$ of $A_+$, $A[\alpha,\beta]$ contains at most one nonzero entry, which is $-$ (because $A[\alpha,\beta]$ contains no $+$ for any two distinct irreducible components $\alpha$ and $\beta$ of $A_+$). Let $X$ be a sign pattern matrix of order $m$ given by
$$x_{ij}=\begin{cases}
-,&\mbox{if }a_{pq}=-\mbox{ for some }p\in\alpha_i\mbox{ and }q\in\alpha_j;\\
0,&\mbox{otherwise}.
\end{cases}$$ 
Since $A$ is minimally AP-irreducible, $X$ is minimally irreducible. Therefore by Lemma \ref{lem3.10}, $X$ has at most $2m-2$ nonzero entries, $A[\alpha_i]$ has at most $2n_i-2$ nonzero entries for $i=1,2,\ldots,m$. Therefore the total number of nonzero entries of $A$ is at most $(2m-2)+(2n_1-2)+\cdots+(2n_m-2)=2n-2$.
\end{proof}

\begin{lemma}\label{lem3.12}
Let $A$ be a minimally AP-irreducible sign pattern matrix with all diagonal entries equal to $0$ such that $A[\alpha,\beta]$ contains no $+$ for any two distinct irreducible components $\alpha$ and $\beta$ of $A_+$. Suppose that $A$ allows an algebraically positive matrix with a simple positive eigenvalue and corresponding left, right positive eigenvectors. If $X$ is a super-pattern of $A$ such that $x_{ij}\neq a_{ij}$ for precisely one ordered pair $(i,j)$, then $X$ also allows an algebraically positive matrix with a simple positive eigenvalue and corresponding left, right positive eigenvectors. 
\end{lemma}

\begin{proof}
Since $A$ is a minimally AP-irreducible and allows a matrix with a positive eigenvalue and a corresponding positive right eigenvector, by Lemma \ref{lem3.11}, $A$ has at least two rows that contains precisely one nonzero entry as $+$. Then $X$ has at least one row that contains precisely one nonzero entry as $+$. Therefore by Theorem \ref{thm1.1} and \ref{thm3.7}, $X$ allows an algebraically positive matrix with a simple positive eigenvalue and corresponding left, right positive eigenvectors. 
\end{proof}

\begin{lemma}\label{lem3.13}
Let $a_{12}\neq0$. If $\lambda$ is a simple eigenvalue of
$$A=\left[\begin{array}{cc|c}
0&a_{12}&\0^T \\
a_{21}&0&\Y^T \\
\hline
\X&\0&R
\end{array}\right]$$
such that $A\U=\lambda\U$ and $\V^TA=\lambda \V^T$ for some $\U,\V>\0$, then $\lambda$ is a simple eigenvalue of
$$B=\left[\begin{array}{c|c}
\lambda+a_{21}-\frac{\lambda^2}{a_{12}}&\Y^T \\[1 mm]
\hline
\X&R
\end{array}\right].$$
Moreover, a right eigenvector $\W$ and a left eigenvector $\Z$ of $B$ corresponding to the eigenvalue $\lambda$ are given by:
$$w_i=\begin{cases} 
u_1,&\mbox{if }i=1;\\
u_{i+1},&\mbox{if }i\geq2.
\end{cases}\qquad \qquad
z_i=\begin{cases} 
v_2,&\mbox{if }i=1;\\
v_{i+1},&\mbox{if }i\geq2.
\end{cases}$$
\end{lemma}

\begin{proof}
It is easy to verify that $B\W=\lambda\W$ and $\Z^TB=\lambda\Z^T$. By elementary row and column operations, we have the following:
$$A-\lambda I=\left[\begin{array}{cc|c}
-\lambda&a_{12}&\0^T \\
a_{21}&-\lambda&\Y^T \\
\hline
\X&\0&R-\lambda I
\end{array}\right]\longrightarrow
\left[\begin{array}{cc|c}
a_{12}&0&\0^T \\
0&a_{21}-\frac{\lambda^2}{a_{12}}&\Y^T \\[1 mm]
\hline
\0&\X&R-\lambda I
\end{array}\right]
=\left[\begin{array}{c|cc}
a_{12}&\0^T \\
\hline
\0&B-\lambda I
\end{array}\right]$$
Since $a_{12}\neq0$, $\rk(A-\lambda I)=\rk(B-\lambda I)+1$. Therefore by Lemma \ref{lem2.1}, we can conclude that $\lambda$ is a simple eigenvalue of $B$.
\end{proof}

\begin{remark}\label{rem3.14}
If $\lambda>0$ and all nonzero entries of $\Y$ are negative in Lemma \ref{lem3.13}, then $\lambda+a_{21}-\frac{\lambda^2}{a_{12}}>0$.
\end{remark}

The following result gives a sufficient condition for a sign pattern matrix to allow algebraic positivity.

\begin{theorem}\label{thm3.16}
Let $A$ be an AP-irreducible sign pattern matrix of order $n$. If $A[\alpha,\beta]$ contains no $+$ for any two distinct irreducible components $\alpha$ and $\beta$ of $A_+$, then $A$ allows algebraic positivity.
\end{theorem}

\begin{proof}
Since $A$ is AP-irreducible, every row and column of $A$ contains a $+$. Then according to our hypothesis, $A[\alpha]=[+]$ for all singleton irreducible component $\alpha$ of $A_+$. Let $\alpha_1,\alpha_2,\ldots,\alpha_{m_r}$ be the distinct irreducible components of $A_+$. Let $X$ be a minimally AP-irreducible subpattern of $A$ such that $\alpha_1,\alpha_2,\ldots,\alpha_{m_r}$ are also distinct irreducible components of $X_+$. Then $X[\alpha_i]=X_+[\alpha_i]$ is minimally irreducible for $i=1,2,\ldots,m_r$.  
 
\noindent{\tt Case I:} $X$ has all diagonal entries $0$.
 
Since $X$ is minimally AP-irreducible, for any two distinct irreducible components $\alpha$ and $\beta$ of $X_+$, $X[\alpha,\beta]$ contains at most one nonzero entry, which is $-$. We color the arc $(\overrightarrow{p,q})$ in $D(X)$ blue if both $p,q$ belong to same $\alpha_i$. Otherwise, we color the arc red. Since $B_X$ is irreducible, for every $i\in\{1,2,\ldots,m_r\}$, $\alpha_i$ contains initial vertex of some red arc and terminal vertex of some other red arc. 

We associate with $D(X)$ a digraph $D_0$ having vertices $v_1,v_2,\ldots,v_{m_r}$ such that $D_0$ has no loop and $(\overrightarrow{v_i,v_j})$ is an arc of $D_0$ if and only if $(\overrightarrow{p,q})$ is a red arc in $D(X)$ for some $p\in \alpha_i$ and $q\in \alpha_j$. Since $X$ is minimally AP-irreducible, $D_0$ is minimally strongly connected. So by Lemma \ref{lem3.8}, there exists a nested sequence $V_{0,1},V_{0,2},\ldots,V_{0,r}$ of vertex sets satisfying the properties P1, P2, P3 such that $V_{0,1}\subsetneq V_{0,2}\subsetneq \cdots\subsetneq V_{0,r}=V(D_0)$. Without loss of generality, we may assume that $V_{0,i}=\{v_1,v_2,\ldots,v_{m_i}\}$ for $i=1,2,\ldots,r$. Now we give a constructive method to show that $X$ allows algebraic positivity. 
\paragraph{Step 1.} Since the subgraph of $D_0$ induced by $V_{0,1}$ is a directed cycle, let $(\overrightarrow{v_1, v_2}),\ldots,(\overrightarrow{v_{m_1-1},v_{m_1}}),(\overrightarrow{v_{m_1},v_1})$ be the arcs of that directed cycle. Let $B_{[1,0]}$ be the matrix of order $m_1$ with index set $\{v_1,v_2,\ldots,v_{m_1}\}$, defined by
$$\left(B_{[1,0]}\right)_{ij}=\begin{cases}
+,&\mbox{if }i=j; \\
-,&\mbox{if }(i,j)\in\{(v_1, v_2),\ldots,(v_{m_1-1},v_{m_1}),(v_{m_1},v_1)\}; \\
0,&\mbox{otherwise}.
\end{cases}$$
Since $B_{[1,0]}$ is irreducible and all nonzero off-diagonal entries of $B_{[1,0]}$ are $-$, by Lemma \ref{lem3.6}, $B_{[1,0]}$ allows algebraic positivity. Since every diagonal entry of $B_{[1,0]}$ is $+$, by Theorem \ref{thm1.1}, we can find an algebraically positive matrix in $Q(B_{[1,0]})$ with a simple positive eigenvalue and corresponding left, right positive eigenvectors.
\paragraph{Step 2.} Let $(\overrightarrow{p_1,q_2}),(\overrightarrow{p_2,q_3}),\ldots,(\overrightarrow{p_{m_1-1},q_{m_1}}),(\overrightarrow{p_{m_1},q_1})$ be the red arcs in $D(X)$, where $p_s,q_s\in\alpha_s$ for $s=1,2,\ldots,m_1$. Suppose that for some $s\in\{1,2,\ldots,m_1\}$, we have constructed a sign pattern matrix $B_{[1,s-1]}$ with index set $\alpha_1\cup\cdots\cup\alpha_{s-1}\cup\{v_s,\ldots,v_{m_1}\}$, which is defined by
$$\left( B_{[1,s-1]}\right)_{ij}=\begin{cases}
\left(B_{[1,0]}\right)_{ij},&\mbox{if }i,j\in\{v_s,\ldots,v_{m_1}\};\\
x_{ij},&\mbox{if }i,j\in\alpha_1\cup\cdots\cup\alpha_{s-1}; \\
-,&\mbox{if }(i,j)\in\{(p_{s-1},v_s),(v_{m_1},q_1)\}; \\
0,&\mbox{otherwise}
\end{cases}$$
allows an algebraically positive matrix with a simple positive eigenvalue and corresponding left, right positive eigenvectors. 

Now we construct a sign pattern matrix $B_{[1,s]}$ that allows an algebraically positive matrix with a simple positive eigenvalue and corresponding left, right positive eigenvectors such that 
$$\left( B_{[1,s]}\right)_{ij}=x_{ij}\mbox{ for all }i,j\in\alpha_1\cup\cdots\cup\alpha_s.$$
\begin{enumerate}
\item Suppose that the subgraph of $D(X)$ induced by $\alpha_s$ is a cycle. We construct a sign pattern matrix $B_{[1,s]}$ with index set $\alpha_1\cup\cdots\cup\alpha_s\cup\{v_{s+1},\ldots,v_{m_1}\}$ as follows:
$$\left( B_{[1,s]}\right)_{ij}=\begin{cases}
\left( B_{[1,s-1]}\right)_{ij},&\mbox{if }i,j\in\alpha_1\cup\cdots\cup\alpha_{s-1}\cup\{v_{s+1},\ldots,v_{m_1}\};\\
+,&\mbox{if }(\overrightarrow{i,j})\mbox{ is a blue arc in }D(X)\mbox{ with }i,j\in\alpha_s; \\
-,&\mbox{if }(i,j)\in\{(p_{s-1},q_s),(p_s,v_{s+1}),(v_{m_1},q_1)\}; \\
0,&\mbox{otherwise}.
\end{cases}$$ 
By Theorem \ref{thm2.4}.1 and Lemma \ref{lem3.5}, $B_{[1,s]}$ allows an algebraically positive matrix with a simple positive eigenvalue and corresponding left, right positive eigenvectors. Moreover, $$\left( B_{[1,s]}\right)_{ij}=x_{ij}\mbox{ for all }i,j\in\alpha_1\cup\cdots\cup\alpha_s.$$
\item Suppose that the subgraph of $D(X)$ induced by $\alpha_s$ is not a cycle. Since $X[\alpha_s]=X_+[\alpha_s]$ is minimally irreducible, the subgraph of $D(X)$ induced by $\alpha_s$ is minimally strongly connected and thus by Lemma \ref{lem3.8}, there exists a nested sequence $V_{s,1},V_{s,2},\ldots,V_{s,k_s}$ of vertex sets satisfying the properties P1, P2, P3 such that $$p_s\in V_{s,1}\subsetneq V_{s,2}\subsetneq \cdots\subsetneq V_{s,k_s}=\alpha_s.$$ 
For every $t\in\{1,2,\ldots,k_s-1\}$, let $p_{s,t},q_{s,t}\in V_{s,t}$ be the initial and the terminal vertices, respectively, of blue arcs whose other end vertices are in $V_{s,t+1}\setminus V_{s,t}$. The subgraph of $D(X)$ induced by $\alpha_s$ is minimally strongly connected and thus the arc $(\overrightarrow{p_{s,t},q_{s,t}})$ does not exist. Equivalently, the $(p_{s,t},q_{s,t})$-th entry of $X$ is 0.
\begin{enumerate}
\item[A.] Let $q_s\in V_{s,1}$. We construct a sign pattern matrix $B_{[1,s-1],1}$ with index set $V_{s,1}\cup(\alpha_1\cup\cdots\cup\alpha_{s-1})\cup\{v_{s+1},\ldots,v_{m_1}\}$ as follows:
$$\left( B_{[1,s-1],1}\right)_{ij}=\begin{cases}
\left( B_{[1,s-1]}\right)_{ij},&\mbox{if }i,j\in\alpha_1\cup\cdots\cup\alpha_{s-1}\cup\{v_{s+1},\ldots,v_{m_1}\};\\
+,&\mbox{if }(i,j)=(p_{s,1},q_{s,1})\mbox{ or }\\
&\qquad(\overrightarrow{i,j})\mbox{ is a blue arc in }D(X)\mbox{ with }i,j\in V_{s,1}; \\
-,&\mbox{if }(i,j)\in\{(p_{s-1},q_s),(p_s,v_{s+1}),(v_{m_1},q_1)\}; \\
0,&\mbox{otherwise}.
\end{cases}$$ 
Let $\tilde{B}_{[1,s-1],1}$ be the subpattern of $B_{[1,s-1],1}$ obtained by replacing the $(p_{s,1},q_{s,1})$-th entry with 0. Then by Lemma \ref{lem3.5} and Theorem \ref{thm2.4}.1, $\tilde{B}_{[1,s-1],1}$ allows an algebraically positive matrix with simple positive eigenvalue and corresponding left right positive eigenvectors.

If $p_{s,1}=p_s=q_s$, then $\tilde{B}_{[1,s-1],1}$ has at least one row with precisely one nonzero entry, which is $+$ (because the subgraph of $D(X)$ induced by $V_{s,1}$ is a directed cycle of length at least 2 with no red arc). Otherwise, either the $p_{s,1}$-th row or the $p_{s,1}$-th column of $\tilde{B}_{[1,s-1],1}$ contains precisely one nonzero entry, which is $+$. Therefore by Theorem \ref{thm3.7}, $B_{[1,s-1],1}$ allows an algebraically positive matrix with simple positive eigenvalue and corresponding left right positive eigenvectors. Moreover, $$\left( B_{[1,s-1],1}\right)_{ij}=x_{ij}\mbox{ for all }i,j\in V_{s,1}\cup(\alpha_1\cup\cdots\cup\alpha_{s-1})\mbox{ satisfying }(i,j)\neq(p_{s,1},q_{s,1}).$$ 
\begin{enumerate}
\item[A.1.] For every $t\in\{2,3,\ldots,k_s-1\}$, after obtaining $B_{[1,s-1],t-1}$, we construct a sign pattern matrix $B_{[1,s-1],t}$ with index set $V_{s,t}\cup(\alpha_1\cup\cdots\cup\alpha_{s-1})\cup\{v_{s+1},\ldots,v_{m_1}\}$ as follows:
$$\left( B_{[1,s-1],t}\right)_{ij}=\begin{cases}
\left( B_{[1,s-1],t-1}\right)_{ij},&\mbox{if }i,j\in V_{s,t-1}\cup\alpha_1\cup\cdots\cup\alpha_{s-1}\cup\{v_{s+1},\ldots,v_{m_1}\}\\
&\hspace{3.3 cm}\mbox{with }(i,j)\neq(p_{s,t-1},q_{s,t-1});\\
+,&\mbox{if }(i,j)=(p_{s,t},q_{s,t}),\mbox{ or }(\overrightarrow{i,j})\mbox{ is a blue arc in }D(X)\\
&\qquad\quad\mbox{with }i,j\in (V_{s,t}\setminus V_{s,t-1})\cup\{p_{s,t-1},q_{s,t-1}\}; \\
0,&\mbox{otherwise}.
\end{cases}$$
Now by Lemma \ref{lem3.5} and one of Theorem \ref{thm2.5}.2 and \ref{thm2.5}.3, $B_{[1,s-1],t}$ allows an algebraically positive matrix satisfying the hypothesis of Theorem \ref{thm2.5}. Moreover, $$\left( B_{[1,s-1],t}\right)_{ij}=x_{ij}\mbox{ for all }i,j\in V_{s,t}\cup(\alpha_1\cup\cdots\cup\alpha_{s-1})\mbox{ satisfying }(i,j)\neq(p_{s,t},q_{s,t}).$$ 
After obtaining $B_{[1,s-1],k_s-1}$, we construct a sign pattern matrix $B_{[1,s-1],k_s}$ with index set $V_{s,k_s}\cup(\alpha_1\cup\cdots\cup\alpha_{s-1})\cup\{v_{s+1},\ldots,v_{m_1}\}$ as follows:
$$\left( B_{[1,s-1],k_s}\right)_{ij}=\begin{cases}
\left( B_{[1,s-1],k_s-1}\right)_{ij},&\mbox{if }i,j\in V_{s,k_s-1}\cup\alpha_1\cup\cdots\cup\alpha_{s-1}\cup\{v_{s+1},\ldots,v_{m_1}\}\\
&\hspace*{2.8 cm}\mbox{with }(i,j)\neq(p_{s,k_s-1},q_{s,k_s-1});\\
+,&\mbox{if }(\overrightarrow{i,j})\mbox{ is a blue arc in }D(X)\mbox{ with}\\
&\hspace*{1.5 cm} i,j\in (V_{s,k_s}\setminus V_{s,k_s-1})\cup\{p_{s,k_s-1},q_{s,k_s-1}\}; \\
0,&\mbox{otherwise}.
\end{cases}$$
Set $B_{[1,s]}=B_{[1,s-1],k_s}$. By Theorem \ref{thm2.5}.1 and Lemma \ref{lem3.5}, $B_{[1,s]}$ allows an algebraically positive matrix with a simple positive eigenvalue and corresponding left, right positive eigenvectors. Since $V_{s,k_s}=\alpha_s$, 
$$\left( B_{[1,s]}\right)_{ij}=x_{ij}\mbox{ for all }i,j\in \alpha_1\cup\cdots\cup\alpha_s.$$ 
\end{enumerate}
\item[B.] Let $q_s\in V_{s,k}$ for some $k$ satisfying $2\leq k\leq k_s$ such that $q_s\notin V_{s,k-1}$. We construct a sign pattern matrix $B_{[1,s-1],1}$ with index set $V_{s,1}\cup(\alpha_1\cup\cdots\cup\alpha_{s-1})\cup\{v_{s+1},\ldots,v_{m_1}\}$ as follows:
$$\left( B_{[1,s-1],1}\right)_{ij}=\begin{cases}
\left( B_{[1,s-1]}\right)_{ij},&\mbox{if }i,j\in\alpha_1\cup\cdots\cup\alpha_{s-1}\cup\{v_{s+1},\ldots,v_{m_1}\};\\
+,&\mbox{if }(i,j)=(p_{s,1},q_{s,1})\mbox{ or }\\
&\hspace*{1 cm}(\overrightarrow{i,j})\mbox{ is a blue arc in }D(X)\mbox{ with }i,j\in V_{s,1}; \\
-,&\mbox{if }(i,j)\in\{(p_{s-1},q_{s,1}),(p_s,v_{s+1}),(v_{m_1},q_1)\}; \\
0,&\mbox{otherwise}.
\end{cases}$$ 
Now by Lemma \ref{lem3.5} and one of Theorem \ref{thm2.4}.2 and \ref{thm2.4}.3, $B_{[1,s-1],1}$ allows an algebraically positive matrix satisfying the hypothesis of Theorem \ref{thm2.5}. Moreover, $$\left( B_{[1,s-1],1}\right)_{ij}=x_{ij}\mbox{ for all }i,j\in V_{s,1}\cup(\alpha_1\cup\cdots\cup\alpha_{s-1})\mbox{ with }(i,j)\notin\{(p_{s,1},q_{s,1}),(p_{s-1},q_{s,1})\}.$$ 

For every $t\in\{2,3,\ldots,k-1\}$, after obtaining $B_{[1,s-1],t-1}$, we construct a sign pattern matrix $B_{[1,s-1],t}$ with index set $V_{s,t}\cup(\alpha_1\cup\cdots\cup\alpha_{s-1})\cup\{v_{s+1},\ldots,v_{m_1}\}$ as follows:
$$\left( B_{[1,s-1],t}\right)_{ij}=\begin{cases}
\left( B_{[1,s-1],t-1}\right)_{ij},&\mbox{if }i,j\in V_{s,t-1}\cup\alpha_1\cup\cdots\cup\alpha_{s-1}\cup\{v_{s+1},\ldots,v_{m_1}\}\\
&\hspace*{1.3 cm}\mbox{with }(i,j)\notin\{(p_{s,t-1},q_{s,t-1}),(p_{s-1},q_{s,t-1})\};\\
+,&\mbox{if }(i,j)=(p_{s,t},q_{s,t}),\mbox{ or }(\overrightarrow{i,j})\mbox{ is a blue arc in }D(X)\\
&\hspace*{1.5 cm}\mbox{with } i,j\in (V_{s,t}\setminus V_{s,t-1})\cup\{p_{s,t-1},q_{s,t-1}\}; \\
-,&\mbox{if }(i,j)=(p_{s-1},q_{s,t});\\
0,&\mbox{otherwise}.
\end{cases}$$
Now by Lemma \ref{lem3.5} and one of Theorem \ref{thm2.5}.2 and \ref{thm2.5}.3, $B_{[1,s-1],t}$ allows an algebraically positive matrix satisfying the hypothesis of Theorem \ref{thm2.5}. Moreover, $$\left( B_{[1,s-1],t}\right)_{ij}=x_{ij}\mbox{ for all }i,j\in V_{s,t}\cup(\alpha_1\cup\cdots\cup\alpha_{s-1})\mbox{ with }(i,j)\notin\{(p_{s,t},q_{s,t}),(p_{s-1},q_{s,t})\}.$$ 
\begin{enumerate}
\item[B.1.] If $k=k_s$, then we construct a sign pattern matrix $B_{[1,s-1],k}$ with index set $V_{s,k}\cup(\alpha_1\cup\cdots\cup\alpha_{s-1})\cup\{v_{s+1},\ldots,v_{m_1}\}$ as follows:
$$\left( B_{[1,s-1],k}\right)_{ij}=\begin{cases}
\left( B_{[1,s-1],k-1}\right)_{ij},&\mbox{if }i,j\in V_{s,k-1}\cup\alpha_1\cup\cdots\cup\alpha_{s-1}\cup\{v_{s+1},\ldots,v_{m_1}\}\\
&\hspace*{1 cm}\mbox{with }(i,j)\notin\{(p_{s,k-1},q_{s,k-1}),(p_{s-1},q_{s,k-1})\};\\
+,&\mbox{if }(\overrightarrow{i,j})\mbox{ is a blue arc in }D(X)\mbox{ with }\\
&\hspace*{2.7 cm} i,j\in (V_{s,k}\setminus V_{s,k-1})\cup\{p_{s,k-1},q_{s,k-1}\}; \\
-,&\mbox{if }(i,j)=(p_{s-1},q_s);\\
0,&\mbox{otherwise}.
\end{cases}$$
Set $B_{[1,s]}=B_{[1,s-1],k}$. By Theorem \ref{thm2.5}.1 and Lemma \ref{lem3.5}, $B_{[1,s]}$ allows an algebraically positive matrix with a simple positive eigenvalue and corresponding left, right positive eigenvectors. Since $V_{s,k}=\alpha_s$, 
$$\left( B_{[1,s]}\right)=x_{ij}\mbox{ for all }i,j\in\alpha_1\cup\cdots\cup\alpha_s.$$
\item[B.2.] If $k<k_s$, then we construct a sign pattern matrix $B_{[1,s-1],k}$ with index set $V_{s,k}\cup(\alpha_1\cup\cdots\cup\alpha_{s-1})\cup\{v_{s+1},\ldots,v_{m_1}\}$ as follows:
$$\left( B_{[1,s-1],k}\right)_{ij}=\begin{cases}
\left( B_{[1,s-1],k-1}\right)_{ij},&\mbox{if }i,j\in V_{s,k-1}\cup\alpha_1\cup\cdots\cup\alpha_{s-1}\cup\{v_{s+1},\ldots,v_{m_1}\}\\
&\hspace*{0.8 cm}\mbox{with }(i,j)\notin\{(p_{s,k-1},q_{s,k-1}),(p_{s-1},q_{s,k-1})\};\\
+,&\mbox{if }(i,j)=(p_{s,k},q_{s,k}),\mbox{ or }(\overrightarrow{i,j})\mbox{ is a blue arc in }D(X)\\
&\hspace*{1 cm}\mbox{with } i,j\in (V_{s,k}\setminus V_{s,k-1})\cup\{p_{s,k-1},q_{s,k-1}\}; \\
-,&\mbox{if }(i,j)=(p_{s-1},q_s);\\
0,&\mbox{otherwise}.
\end{cases}$$
Let $\tilde{B}_{[1,s-1],k}$ be the subpattern of $B_{[1,s-1],k}$ obtained by replacing the $(p_{s,k},q_{s,k})$-th entry with 0. Then by Lemma \ref{lem3.5} and Theorem \ref{thm2.4}.1, $\tilde{B}_{[1,s-1],k}$ allows an algebraically positive matrix with simple positive eigenvalue and corresponding left right positive eigenvectors.

If $p_{s,k}=p_s=q_s$, then $\tilde{B}_{[1,s-1],k}$ has at least one row with no $-$ (because the subgraph of $D(X)$ induced by $V_{s,k}$ has no red arc and it has at least two vertices). Otherwise, either the $p_{s,k}$-th row or the $p_{s,k}$-th column of $\tilde{B}_{[1,s-1],k}$ contains no $-$. Therefore by Theorem \ref{thm3.7}, $B_{[1,s-1],k}$ allows an algebraically positive matrix with simple positive eigenvalue and corresponding left right positive eigenvectors. Moreover, $$\left( B_{[1,s-1],k}\right)_{ij}=x_{ij}\mbox{ for all }i,j\in V_{s,k}\cup(\alpha_1\cup\cdots\cup\alpha_{s-1})\mbox{ with }(i,j)\neq(p_{s,k},q_{s,k}).$$ 
Now we apply the procedure in A.1. to construct $B_{[1,s-1],k_s}$. Set $B_{[1,s]}=B_{[1,s-1],k_s}$. Then $B_{[1,s]}$ allows an algebraically positive matrix with a simple positive eigenvalue and corresponding left, right positive eigenvectors such that   
$$\left( B_{[1,s]}\right)_{ij}=x_{ij}\mbox{ for all }i,j\in \alpha_1\cup\cdots\cup\alpha_s.$$ 
\end{enumerate}
\end{enumerate}
However, whether $\alpha_s$ is a directed cycle or not, we have $B_{[1,s]}$ as follows:
$$\left( B_{[1,s]}\right)_{ij}=\begin{cases}
\left( B_{[1,0]}\right)_{ij},&\mbox{if }i,j\in\{v_{s+1},\ldots,v_{m_1}\};\\
x_{ij},&\mbox{if }i,j\in\alpha_1\cup\cdots\cup\alpha_s; \\
-,&\mbox{if }(i,j)\in\{(p_s,v_{s+1}),(v_{m_1},q_1)\}; \\
0,&\mbox{otherwise}
\end{cases}$$
\end{enumerate}
When $s=m_1$, then we have a sign pattern matrix $B_{[1,m_1]}$ with index set $\alpha_1\cup\cdots\cup\alpha_{m_1}$ that allows an algebraically positive matrix with a simple positive eigenvalue and corresponding left, right positive eigenvectors such that $$\left( B_{[1,m_1]}\right)_{ij}=x_{ij}\mbox{ for all }i,j\in \alpha_1\cup\cdots\cup\alpha_{m_1}.$$ 
\paragraph{Step 3.} Let $m_0=0$. Suppose that we have constructed $B_{[t,m_t-m_{t-1}]}$ for some $t\in\{1,2,\ldots,r-1\}$ that allows an algebraically positive matrix with a simple positive eigenvalue and corresponding left, right positive eigenvectors such that
$$\left( B_{[t,m_t-m_{t-1}]}\right)_{ij}=x_{ij}\mbox{ for all }i,j\in \alpha_1\cup\cdots\cup\alpha_{m_t}.$$
In $D_0$, we have $V_{0,t+1}\setminus V_{0,t}=\{v_{m_t+1},\ldots,v_{m_{t+1}}\}$. In accordance with Remark \ref{rem3.9}, let $p_0,q_0\in\alpha_1\cup\cdots\cup\alpha_{m_t}$ be such that $p_0$ is the initial vertex of a red arc in $D(X)$ with the terminal vertex in $\alpha_p$ and $q_0$ is the terminal vertex of a red arc in $D(X)$ with the initial vertex in $\alpha_q$ for some $p,q\in\{m_t+1,m_t+2,\ldots,m_{t+1}\}$. Now we construct $B_{[t+1,0]}$ with index set $\alpha_1\cup\cdots\cup\alpha_{m_t}\cup\{v_{m_t+1},\ldots,v_{m_{t+1}}\}$ as follows:
$$\left( B_{[t+1,0]}\right)_{ij}=\begin{cases}
x_{ij},&\mbox{if }i,j\in\alpha_1\cup\cdots\cup\alpha_{m_t}; \\
+,&\mbox{if }i=j\in\{v_{m_t+1},\ldots,v_{m_{t+1}}\};\\
-,&\mbox{if }(i,j)\in\{(p_0,v_p),(v_q,q_0)\},\mbox{ or }(\overrightarrow{i,j})\mbox{ is an arc in}\\
&\hspace*{3 cm}\mbox{the subgraph of } D_0 \mbox{ induced by }V_{0,t+1}\setminus V_{0,t};\\
0,&\mbox{otherwise}
\end{cases}$$
Let $p_0\in\alpha_{\tilde{p}}$ and $q_0\in\alpha_{\tilde{q}}$ for some $\tilde{p},\tilde{q}\in\{1,2,\ldots,m_t\}$. Since $D_0$ is minimally strongly connected and it has the arcs $(\overrightarrow{v_{\tilde{p}},v_p})$ and $(\overrightarrow{v_q,v_{\tilde{q}}})$, by Remark \ref{rem3.9}, $D_0$ does not have the arc $(\overrightarrow{v_{\tilde{p}},v_{\tilde{q}}})$. So by the definition of $D_0$, either $(\overrightarrow{p_0,q_0})$ is a blue arc in $D(X)$, or $D(X)$ does not have the arc $(\overrightarrow{p_0,q_0})$.
\begin{enumerate}
\item If $(\overrightarrow{p_0,q_0})$ is a blue arc in $D(X)$, then the $(p_0,q_0)$-th entry of $B_{[t,m_t-m_{t-1}]}$ is $+$. Therefore by Lemma \ref{lem3.5} and Theorem \ref{thm2.3}, $B_{[t+1,0]}$ allows an algebraically positive matrix with a simple positive eigenvalue and corresponding left, right positive eigenvectors.
\item If $D(X)$ does not have the arc $(\overrightarrow{p_0,q_0})$, then the $(p_0,q_0)$-th entry of $B_{[t,m_t-m_{t-1}]}$ is $0$. We construct a super-pattern of $B_{[t,m_t-m_{t-1}]}$ by replacing \enquote*{$0$} with \enquote*{$-$} at $(p_0,q_0)$-th entry. Since $X$ is minimally AP-irreducible, $B_{[t,m_t-m_{t-1}]}$ is minimally AP-irreducible. Again according to our assumption, all diagonal entries of $B_{[t,m_t-m_{t-1}]}$ are 0. So by Lemma \ref{lem3.12}, the above mentioned super-pattern of $B_{[t,m_t-m_{t-1}]}$ allows an algebraically positive matrix satisfying the hypothesis of Theorem \ref{thm2.2}. Therefore by Lemma \ref{lem3.5} and Theorem \ref{thm2.2}, $B_{[t+1,0]}$ allows an algebraically positive matrix with a simple positive eigenvalue and corresponding left, right positive eigenvectors.
\end{enumerate} 
Note that $$\left( B_{[t+1,0]}\right)_{ij}=x_{ij}\mbox{ for all }i,j\in \alpha_1\cup\cdots\cup\alpha_{m_t}.$$
Now we repeat the method in Step 2 to construct a sign pattern matrix $B_{[t+1,m_{t+1}-m_t]}$ with index set $\alpha_1\cup\cdots\cup\alpha_{m_{t+1}}$ which allows an algebraically positive matrix with a simple positive eigenvalue and corresponding left, right positive eigenvectors such that $$\left( B_{[t+1,m_{t+1}-m_t]}\right)_{ij}=x_{ij}\mbox{ for all }i,j\in \alpha_1\cup\cdots\cup\alpha_{m_{t+1}}.$$ 
When $t=r-1$, then we have a sign pattern matrix $B_{[r,m_r-m_{r-1}]}$ with index set $\alpha_1\cup\cdots\cup\alpha_{m_r}$ which allows an algebraically positive matrix with a simple positive eigenvalue and corresponding left, right positive eigenvectors such that $$\left( B_{[r,m_r-m_{r-1}]}\right)_{ij}=x_{ij}\mbox{ for all }i,j\in \alpha_1\cup\cdots\cup\alpha_{m_r}.$$ 
Therefore $B_{[r,m_r-m_{r-1}]}=X$. So by Theorem \ref{thm1.1}, $X$ allows algebraic positivity.

\noindent{\tt Case II:} $X$ has some nonzero diagonal entries.

Since $X$ is minimally AP-irreducible, $x_{ii}\neq-$ for all $i\in\{1,2,\ldots,n\}$. Suppose that $x_{ii}=+$ for $i=1,2,\ldots,k$ (without loss of generality). For every $i\in\{1,2,\ldots,k\}$, delete the row and the column corresponding to the index $i$, and append two rows and two columns corresponding to the indices $i'$ and $i''$ such that
\begin{enumerate}
\item both $(i',i'')$-th entry and $(i'',i')$-th entry are $+$, 
\item for every $j\neq i''$, $(i',j)$-th entry is $0$ and $(j,i')$-th entry is the $(j,i)$-th entry of $X$,
\item for every $j\neq i'$, $(j,i'')$-th entry is $0$ and $(i'',j)$-th entry is the $(i,j)$-th entry of $X$.
\end{enumerate}
Let $Y$ be the new sign pattern matrix. Then $Y$ is minimally AP-irreducible and all diagonal entries of $Y$ are 0, and thus by {\tt Case I}, $Y$ allows an algebraically positive matrix with a simple positive eigenvalue and corresponding left, right positive eigenvectors. Therefore by Lemma \ref{lem3.13} and Remark \ref{rem3.14}, $A$ allows an algebraically positive matrix with a simple positive eigenvalue and corresponding left, right positive eigenvectors. So by Theorem \ref{thm1.1}, $X$ allows algebraic positivity.

Since $A$ is a super-pattern of $X$, by Theorem \ref{thm3.7}, $A$ allows algebraic positivity.
\end{proof}

\section{Example to illustrate the methodology to prove Theorem \ref{thm3.16}}\label{sec4}
Let us consider the sign pattern matrix
$$\begin{blockarray}{cccccccccccc}
&1&2&3&4&5&6&7&8&9&10&11 \\
\begin{block}{c[cccc|cccc|c|cc]}
1&0&+&0&0&0&0&-&0&0&0&0 \\
2&0&0&+&0&0&0&0&0&0&0&0 \\
3&+&0&0&+&0&0&0&0&0&0&0 \\
4&0&+&0&0&0&0&0&0&0&0&0 \\
\cline{2-12}
5&0&0&0&0&0&+&0&+&-&0&- \\
6&0&0&0&0&0&0&+&0&0&0&0 \\
7&0&0&0&0&+&0&0&0&0&0&0 \\
8&0&0&0&0&0&0&+&0&0&0&0 \\
\cline{2-12}
9&0&0&0&-&0&0&0&0&+&0&0 \\
\cline{2-12}
10&0&0&0&-&0&0&0&0&0&0&+ \\
11&0&0&0&0&0&0&0&0&0&+&0 \\
\end{block}
\end{blockarray}=X.$$
Observe that $X$ is AP-irreducible. Now $\alpha=\{9\}$ is an irreducible component of $X_+$ and $X[\alpha]=[+]$. Moreover, there is no more such $\alpha$.

Now delete the 9-th row and the 9-th column, and append two rows and two columns corresponding to the indices $9'$ and $9''$ such that both $(9',9'')$-th entry and $(9'',9')$-th entry are $+$, and
\begin{enumerate}
\item for every $j\neq 9''$, $(9',j)$-th entry is $0$, and $(j,9')$-th entry is $(j,9)$-th entry of $X$,
\item for every $j\neq 9'$, $(j,9'')$-th entry is $0$, and $(9'',j)$-th entry is $(9,j)$-th entry of $X$.
\end{enumerate}
Therefore we have a new sign pattern matrix $Y$ as follows:
$$\begin{blockarray}{ccccccccccccc}
&1&2&3&4&5&6&7&8&9'&9''&10&11 \\
\begin{block}{c[cccc|cccc|cc|cc]}
1&0&+&0&0&0&0&-&0&0&0&0&0 \\
2&0&0&+&0&0&0&0&0&0&0&0&0 \\
3&+&0&0&+&0&0&0&0&0&0&0&0 \\
4&0&+&0&0&0&0&0&0&0&0&0&0 \\
\cline{2-13}
5&0&0&0&0&0&+&0&+&-&0&0&- \\
6&0&0&0&0&0&0&+&0&0&0&0&0 \\
7&0&0&0&0&+&0&0&0&0&0&0&0 \\
8&0&0&0&0&0&0&+&0&0&0&0&0 \\
\cline{2-13}
9'&0&0&0&0&0&0&0&0&0&+&0&0 \\
9''&0&0&0&-&0&0&0&0&+&0&0&0 \\
\cline{2-13}
10&0&0&0&-&0&0&0&0&0&0&0&+ \\
11&0&0&0&0&0&0&0&0&0&0&+&0 \\
\end{block}
\end{blockarray}=Y$$
Observe that $Y$ is minimally AP-irreducible. Draw the digraph $D(Y)$, and color the arc $(\overrightarrow{i,j})$ blue if $y_{ij}=+$, red if $y_{ij}=-$.
\begin{figure}[H]
\centering
\begin{tikzpicture}[scale=0.5]
\tikzset{edge/.style={decoration={markings,mark=at position 0.6 with {\arrow[scale=1.5,>=stealth]{>}}}, postaction={decorate}}}
\draw[thick,blue,edge] (160:7)--(180:8);
\draw[thick,blue,edge] (180:8)--(180:5);
\draw[thick,blue,edge] (180:5)--(160:7);
\draw[thick,blue,edge] (200:7)--(180:8);
\draw[thick,blue,edge] (180:5)--(200:7);
\draw[thick,blue,edge] (340:7)--(0:8);
\draw[thick,blue,edge] (0:8)--(20:7);
\draw[thick,blue,edge] (20:7)--(340:7);
\draw[thick,blue,edge] (340:7)--(0:5);
\draw[thick,blue,edge] (0:5)--(20:7);
\draw[thick,blue,edge] (30:2)..controls(90:0.5)..(150:2);
\draw[thick,blue,edge] (150:2)..controls(90:1.5)..(30:2);
\draw[thick,blue,edge] (225:3)..controls(270:1.5)..(315:3);
\draw[thick,blue,edge] (315:3)..controls(270:2.5)..(225:3);
\draw[thick,red,edge] (340:7)--(30:2);
\draw[thick,red,edge] (150:2)--(200:7);
\draw[thick,red,edge] (160:7)--(20:7);
\draw[thick,red,edge] (225:3)--(200:7);
\draw[thick,red,edge] (340:7)--(315:3);
\foreach \i in {20,160,200,340}{
\filldraw (\i:7) circle (3 pt);}
\foreach \i in {0,180}{
\filldraw (\i:8) circle (3 pt);}
\foreach \i in {30,150}{
\filldraw (\i:2) circle (3 pt);}
\foreach \i in {225,315}{
\filldraw (\i:3) circle (3 pt);}
\foreach \i in {0,180}{
\filldraw (\i:5) circle (3 pt);}
\foreach \i/\j in {160/1,200/4,340/5,20/7}{
\node at (\i:7.5) {\j};}
\foreach \i/\j in {180/2, 0/6}{
\node at (\i:8.5) {\j};}
\foreach \i/\j in {180/3, 0/8}{
\node at (\i:4.5) {\j};}
\foreach \i/\j in {230/10, 310/11}{
\node at (\i:3.5) {\j};}
\foreach \i/\j in {145/9'', 35/9'}{
\node at (\i:2.6) {$\j$};}
\end{tikzpicture}
\caption{A diagram of $D(Y)$}
\end{figure}
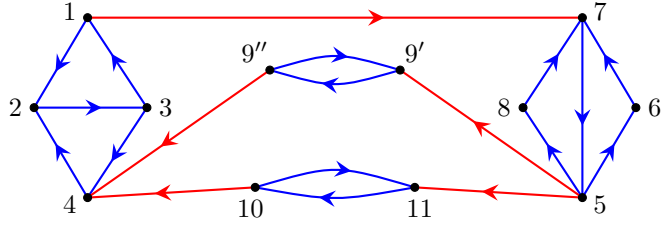
Here for $Y_+$, we have $\alpha_1=\{1,2,3,4\}$, $\alpha_2=\{5,6,7,8\}$, $\alpha_3=\{9',9''\}$ and $\alpha_4=\{10,11\}$. Moreover, the red arcs are $(\overrightarrow{1,7}),(\overrightarrow{5,9'}),(\overrightarrow{9'',4}),(\overrightarrow{10,4}),(\overrightarrow{5,11})$. Associate $D(Y)$ with $D_0$, given by the following diagram. 
\begin{figure}[H]
\centering
\begin{tikzpicture}[scale=0.8]
\tikzset{edge/.style={decoration={markings,mark=at position 0.6 with {\arrow[scale=1.5,>=stealth]{>}}}, postaction={decorate}}}
\draw[thick,edge] (4,0)--(0,0);
\draw[thick,edge] (0,0)--(-4,0);
\draw[thick,edge] (-4,0)..controls(0,1)..(4,0);
\draw[thick,edge] (4,0)--(0,-1);
\draw[thick,edge] (0,-1)--(-4,0);
\foreach \i/\j in {4/0,-4/0,0/-1,0/0}{
\filldraw (\i,\j) circle (2 pt);}
\foreach \i/\j/\k in {4.4/0/v_2,-4.4/0/v_1,0/0.4/v_3,0/-0.7/v_4}{
\node at (\i,\j) {$\k$};}
\end{tikzpicture}
\caption{A diagram of $D_0$}
\end{figure}
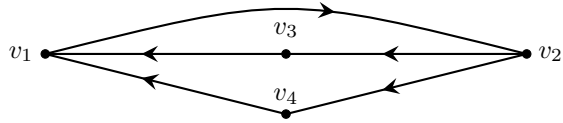
A nested sequence of vertex sets of $D_0$ satisfying the properties P1, P2, P3 of Lemma \ref{lem3.8} is given by: 
$$V_{0,1}=\{v_1,v_2,v_3\}\qquad V_{0,2}=\{v_1,v_2,v_3,v_4\}$$
Construct $B_{[1,0]}$ in accordance with $V_{0,1}$ as follows:
$$\begin{blockarray}{cccc}
    & v_1 & v_2 & v_3 \\
\begin{block}{c[c|cc]}
v_1 & + & - & 0 \\
\cline{2-4} 
v_2 & 0 & + & - \\
v_3 & - & 0 & + \\
\end{block}
\end{blockarray}=B_{[1,0]}$$
Since $B_{[1,0]}$ is irreducible and all nonzero off-diagonal entries of $B_{[1,0]}$ are $-$, by Lemma \ref{lem3.6}, $B_{[1,0]}$ allows algebraic positivity. Since all diagonal entries of $B_{[1,0]}$ are $+$, by Theorem \ref{thm1.1}, $B_{[1,0]}$ allows an algebraically positive matrix with a simple positive eigenvalue and corresponding left, right positive eigenvectors.

Consider the subgraph of $D(Y)$ induced by $\alpha_1=\{1,2,3,4\}$, which is not a directed cycle. Now $(\overrightarrow{1,7})$ is a red arc such that $1\in\alpha_1$ and $7\in\alpha_2$. So we consider a nested sequence of vertex sets of the subgraph of $D(Y)$ induced by $\alpha_1$, as follows, that satisfy the properties P1, P2, P3 of Lemma \ref{lem3.8}:
$$V_{1,1}=\{1,2,3\}\qquad V_{1,2}=\{1,2,3,4\}$$
Observe that $(\overrightarrow{3,4})$ and $(\overrightarrow{4,2})$ are blue arcs in $D(Y)$, where $2,3\in V_{1,1}$ and $4\in V_{1,2}\setminus V_{1,1}$. Further, $(\overrightarrow{9'',4})$ is a red arc, where $9''\in\alpha_3$ and $4\in V_{1,2}\setminus V_{1,1}$. So we take $(3,2)$-th entry as $+$, and $(1,v_2)$-th entry and $(v_3,2)$-th entry as $-$ to construct $B_{[1,0],1}$ as follows. 
$$\begin{blockarray}{cccccc}
&1&2&3&v_2&v_3 \\
\begin{block}{c[ccc|cc]}
1&0&+&0&-&0 \\
2&0&0&+&0&0 \\
3&+&+&0&0&0 \\
\cline{2-6}
v_2&0&0&0&+&- \\
v_3&0&-&0&0&+ \\
\end{block}
\end{blockarray}=B_{[1,0],1}$$
By Lemma \ref{lem3.5} and Theorem \ref{thm2.4}.2, $B_{[1,0],1}$ allows an algebraically positive matrix satisfying the hypothesis of Theorem \ref{thm2.5}. Since $V_{1,2}=\alpha_1$, we take $(3,2)$-th entry as $0$ and $(v_3,4)$-th entry as $-$ to construct $B_{[1,0],2}$ as follows. 

$$\begin{blockarray}{ccccccc}
&1&2&3&4&v_2&v_3 \\
\begin{block}{c[ccc|c|cc]}
1&0&+&0&0&-&0 \\
2&0&0&+&0&0&0 \\
3&+&0&0&+&0&0 \\
\cline{2-7}
4&0&+&0&0&0&0 \\
\cline{2-7}
v_2&0&0&0&0&+&- \\
v_3&0&0&0&-&0&+ \\
\end{block}
\end{blockarray}=B_{[1,0],2}$$
By Lemma \ref{lem3.5} and Theorem \ref{thm2.5}.1, $B_{[1,0],2}$ allows an algebraically positive matrix with a simple positive eigenvalue and corresponding left, right positive eigenvectors. Set $B_{[1,1]}=B_{[1,0],2}$.

Consider the subgraph of $D(Y)$ induced by $\alpha_2=\{5,6,7,8\}$, which is not a directed cycle. Now $(\overrightarrow{5,9'})$ is a red arc, where $5\in\alpha_2$ and $9'\in\alpha_3$. So we consider a nested sequence of vertex sets of the subgraph of $D(Y)$ induced by $\alpha_2$, as follows, that satisfy the properties P1, P2, P3 of Lemma \ref{lem3.8}:
$$V_{2,1}=\{5,6,7\}\qquad V_{2,2}=\{5,6,7,8\}$$
Observe that $(\overrightarrow{5,8})$ and $(\overrightarrow{8,7})$ are blue arcs in $D(Y)$, where $5,7\in V_{2,1}$ and $8\in V_{2,2}\setminus V_{1,1}$. Further, $(\overrightarrow{1,7})$ is a red arc, where $1\in\alpha_1$ and $7\in V_{2,1}$. So we take $(5,7)$-th entry as $+$, and $(1,7)$-th entry and $(5,v_3)$-th entry as $-$ to construct $B_{[1,1],1}$ as follows. 

\begin{align*}
\begin{blockarray}{ccccccccc}
&1&2&3&4&5&6&7&v_3 \\
\begin{block}{c[cccc|ccc|c]}
1&0&+&0&0&0&0&-&0 \\
2&0&0&+&0&0&0&0&0 \\
3&+&0&0&+&0&0&0&0 \\
4&0&+&0&0&0&0&0&0 \\
\cline{2-9}
5&0&0&0&0&0&+&+&- \\
6&0&0&0&0&0&0&+&0 \\
7&0&0&0&0&+&0&0&0 \\
\cline{2-9}
v_3&0&0&0&-&0&0&0&+ \\
\end{block}
\end{blockarray}=B_{[1,1],1}
\end{align*}
By Lemma \ref{lem3.5} and Theorem \ref{thm2.4}.3, $B_{[1,1],1}$ allows an algebraically positive matrix satisfying the hypothesis of Theorem \ref{thm2.5}. Since $V_{2,2}=\alpha_2$, we take $(5,7)$-th entry as $0$ to construct $B_{[1,1],2}$ as follows. 
\begin{align*}
\begin{blockarray}{cccccccccc}
&1&2&3&4&5&6&7&8&v_3 \\
\begin{block}{c[cccc|ccc|c|c]}
1&0&+&0&0&0&0&-&0&0 \\
2&0&0&+&0&0&0&0&0&0 \\
3&+&0&0&+&0&0&0&0&0 \\
4&0&+&0&0&0&0&0&0&0 \\
\cline{2-10}
5&0&0&0&0&0&+&0&+&- \\
6&0&0&0&0&0&0&+&0&0 \\
7&0&0&0&0&+&0&0&0&0 \\
\cline{2-10}
8&0&0&0&0&0&0&+&0&0 \\
\cline{2-10}
v_3&0&0&0&-&0&0&0&0&+ \\
\end{block}
\end{blockarray}=B_{[1,1],2}
\end{align*}
By Lemma \ref{lem3.5} and Theorem \ref{thm2.5}.1, $B_{[1,1],2}$ allows an algebraically positive matrix with a simple positive eigenvalue and corresponding left, right positive eigenvectors. Set $B_{[1,2]}=B_{[1,1],2}$.

Consider the subgraph of $D(Y)$ induced by $\alpha_3=\{9',9''\}$, which is a directed cycle. So we construct $B_{[1,3]}$ as follows. 

\begin{align*}
\begin{blockarray}{ccccccccccc}
&1&2&3&4&5&6&7&8&9'&9'' \\
\begin{block}{c[cccc|cccc|cc]}
1&0&+&0&0&0&0&-&0&0&0 \\
2&0&0&+&0&0&0&0&0&0&0 \\
3&+&0&0&+&0&0&0&0&0&0 \\
4&0&+&0&0&0&0&0&0&0&0 \\
\cline{2-11}
5&0&0&0&0&0&+&0&+&-&0 \\
6&0&0&0&0&0&0&+&0&0&0 \\
7&0&0&0&0&+&0&0&0&0&0 \\
8&0&0&0&0&0&0&+&0&0&0 \\
\cline{2-11}
9'&0&0&0&0&0&0&0&0&0&+ \\
9''&0&0&0&-&0&0&0&0&+&0 \\
\end{block}
\end{blockarray}=B_{[1,3]}
\end{align*}
By Lemma \ref{lem3.5} and Theorem \ref{thm2.4}.1, $B_{[1,3]}$ allows an algebraically positive matrix with a simple positive eigenvalue and corresponding left, right positive eigenvectors.

Consider the set $V_{0,2}=\{v_1,v_2,v_3,v_4\}$. Observe that $(\overrightarrow{5,11})$ and $(\overrightarrow{10,4})$ are arcs in $D(Y)$, where $4,5\in\alpha_1\cup\alpha_2\cup\alpha_3$ and $10,11\in\alpha_4$. So we take $(5,v_4)$-th entry and $(v_4,4)$-th entry as $-$, and $(v_4,v_4)$-th entry as $+$ to construct $B_{[2,0]}$ as follows.

\begin{align*}
\begin{blockarray}{cccccccccccc}
&1&2&3&4&5&6&7&8&9'&9''&v_4 \\
\begin{block}{c[cccc|cccc|cc|c]}
1&0&+&0&0&0&0&-&0&0&0&0 \\
2&0&0&+&0&0&0&0&0&0&0&0 \\
3&+&0&0&+&0&0&0&0&0&0&0 \\
4&0&+&0&0&0&0&0&0&0&0&0 \\
\cline{2-12}
5&0&0&0&0&0&+&0&+&-&0&- \\
6&0&0&0&0&0&0&+&0&0&0&0 \\
7&0&0&0&0&+&0&0&0&0&0&0 \\
8&0&0&0&0&0&0&+&0&0&0&0 \\
\cline{2-12}
9'&0&0&0&0&0&0&0&0&0&+&0 \\
9''&0&0&0&-&0&0&0&0&+&0&0 \\
\cline{2-12}
v_4&0&0&0&-&0&0&0&0&0&0&+ \\
\end{block}
\end{blockarray}=B_{[2,0]}
\end{align*}
Since $(5,4)$-th entry of $B_{[1,3]}$ is $0$, we take a super-pattern of $B_{[1,3]}$ by replacing \enquote*{$0$} with \enquote*{$-$} at $(5,4)$-th entry. By Lemma \ref{lem3.5} and \ref{lem3.12}, that super-pattern allows an algebraically positive matrix with a simple positive eigenvalue and corresponding left, right positive eigenvectors. Therefore by Lemma \ref{lem3.5} and Theorem \ref{thm2.2}, $B_{[2,0]}$ allows an algebraically positive matrix satisfying the hypothesis of Theorem \ref{thm2.4}.

Consider the subgraph of $D(Y)$ induced by $\alpha_4=\{10,11\}$, which is a directed cycle. So we construct $B_{[2,1]}$ as follows. 
$$\begin{blockarray}{ccccccccccccc}
&1&2&3&4&5&6&7&8&9'&9''&10&11 \\
\begin{block}{c[cccc|cccc|cc|cc]}
1&0&+&0&0&0&0&-&0&0&0&0&0 \\
2&0&0&+&0&0&0&0&0&0&0&0&0 \\
3&+&0&0&+&0&0&0&0&0&0&0&0 \\
4&0&+&0&0&0&0&0&0&0&0&0&0 \\
\cline{2-13}
5&0&0&0&0&0&+&0&+&-&0&0&- \\
6&0&0&0&0&0&0&+&0&0&0&0&0 \\
7&0&0&0&0&+&0&0&0&0&0&0&0 \\
8&0&0&0&0&0&0&+&0&0&0&0&0 \\
\cline{2-13}
9'&0&0&0&0&0&0&0&0&0&+&0&0 \\
9''&0&0&0&-&0&0&0&0&+&0&0&0 \\
\cline{2-13}
10&0&0&0&-&0&0&0&0&0&0&0&+ \\
11&0&0&0&0&0&0&0&0&0&0&+&0 \\
\end{block}
\end{blockarray}=B_{[2,1]}$$
By Lemma \ref{lem3.5} and Theorem \ref{thm2.4}.1, $B_{[2,1]}$ allows an algebraically positive matrix with a simple positive eigenvalue and corresponding left, right positive eigenvectors. Moreover, $B_{[2,1]}=Y$. Now by Lemma \ref{lem3.5}, \ref{lem3.13} and Remark \ref{rem3.14}, $X$ allows a matrix with a simple positive eigenvalue and corresponding left, right positive eigenvectors. Therefore by Theorem \ref{thm1.1}, $X$ allows algebraic positivity.

\section{Conclusion}\label{sec5}
We have seen that if $A$ is a minimally AP-irreducible sign pattern matrix of order $n$ with all diagonal entries equal to $0$ such that $A[\alpha,\beta]$ contains no $+$ for any two distinct irreducible components $\alpha$ and $\beta$ of $A_+$, then $A$ has at most $2n-2$ nonzero entries. We have also seen in Section \ref{sec2} that if $A$ allows an algebraically positive matrix with a simple positive eigenvalue, then we can construct a higher-order algebraically positive matrix keeping the condition of AP-irreducibility intact. So we hope that the following results are true.

\begin{conj}
If $A$ is a minimally AP-irreducible sign pattern matrix of order $n$ with all diagonal entries equal to $0$, then $A$ has at most $2n-2$ nonzero entries.
\end{conj}

\begin{conj}\label{conj5.2}
A sign pattern matrix $A$ allows algebraic positivity if and only if either $A$ or $-A$ is AP-irreducible.
\end{conj}

It should be noted that the above two conjectures are valid for all known algebraically positive matrices in \cite{AP19,DB19,KQZ16}.

Let $\mathcal{AP}$ be the set of all sign pattern matrices that allow algebraic positivity. To prove Conjecture \ref{conj5.2}, it is enough to show that all AP-irreducible sign pattern matrices with zero diagonal allow algebraic positivity.

\begin{prop}\label{prop3.15}
If $\mathcal{AP}$ contains all AP-irreducible sign pattern matrices with zero diagonal, then it contains all AP-irreducible sign pattern matrices.
\end{prop}

\begin{proof}
Let $A$ be an AP-irreducible sign pattern matrix of order $n$. Let $X$ be a subpattern of $A$ such that $X$ is minimally AP-irreducible. Since $X$ is minimally AP-irreducible, $x_{ii}\neq-$ for all $i\in\{1,2,\ldots,n\}$. Suppose that $x_{ii}=+$ for $i=1,2,\ldots,k$ (without loss of generality). For every $i\in\{1,2,\ldots,k\}$, delete the row and the column corresponding to the index $i$, and append two rows and two columns corresponding to the indices $i'$ and $i''$ such that
\begin{enumerate}
\item both $(i',i'')$-th entry and $(i'',i')$-th entry are $+$, 
\item for every $j\neq i''$, $(i',j)$-th entry is $0$ and $(j,i')$-th entry is the $(j,i)$-th entry of $A$,
\item for every $j\neq i'$, $(j,i'')$-th entry is $0$ and $(i'',j)$-th entry is the $(i,j)$-th entry of $A$.
\end{enumerate}
Let $Y$ be the new sign pattern matrix. Then $Y$ is minimally AP-irreducible and all diagonal entries of $Y$ are 0, and thus $Y$ allows algebraic positivity. Since $(i',i'')$-th entry is the only nonzero entry in $i'$-th row, $Y$ allows an algebraically positive matrix with a simple positive eigenvalue and corresponding left, right positive eigenvectors. Therefore by Lemma \ref{lem3.13} and Remark \ref{rem3.14}, $X$ allows an algebraically positive matrix with a simple positive eigenvalue and corresponding left, right positive eigenvectors. Then by Theorem \ref{thm1.1} and \ref{thm3.7}, $A$ allows algebraic positivity. Therefore $\mathcal{AP}$ contains all AP-irreducible sign pattern matrices.
\end{proof}

\section*{Acknowledgement}
The author thanks Indian Statistical Institute Delhi for providing him with a visiting fellowship. The author also thanks Dr. Sriparna Bandyopadhyay for her careful reading and suggestions for a better presentation of the paper.

\end{document}